\documentclass{siamart1116}

\usepackage{amssymb,latexsym,amsmath,enumerate,verbatim,amsfonts,cite,lscape,algorithm,algorithmic}
\usepackage{color} 
\usepackage[active]{srcltx}

\ifpdf
  \DeclareGraphicsExtensions{.eps,.pdf,.png,.jpg}
\else
  \DeclareGraphicsExtensions{.eps}
\fi

\numberwithin{theorem}{section}

\newcommand{\TheTitle}{Iteratively reweighted $\ell_1$ algorithms with extrapolation}
\newcommand{\TheAuthors}{Peiran Yu and Ting Kei Pong}

\headers{\TheTitle}{\TheAuthors}

\title{{\TheTitle}}

\author{
  Peiran Yu\thanks{Department of Applied
Mathematics, the Hong Kong Polytechnic University, Hong Kong. \email{peiran.yu.yu@connect.polyu.hk}.}
  \and
  Ting Kei Pong\thanks{Department of Applied Mathematics, the Hong Kong Polytechnic University, Hong Kong.
This author's work was supported in part by Hong Kong Research Grants Council PolyU153085/16p \email{tk.pong@polyu.edu.hk}.}
}

\usepackage{amsopn}

\ifpdf
\hypersetup{
  pdftitle={\TheTitle},
  pdfauthor={\TheAuthors}
}
\fi

\def\Argmin{\mathop{\rm Arg\,min}}
\def\argmin{\mathop{\rm arg\,min}}

\def\dist{\mbox{\rm dist}\,}
\def\dom{\mbox{\rm dom}\,}

\begin{document}

\maketitle

\begin{abstract}
Iteratively reweighted $\ell_1$ algorithm is a popular algorithm for solving a large class of optimization problems whose objective is the sum of a Lipschitz differentiable loss function and a possibly nonconvex sparsity inducing regularizer. In this paper, motivated by the success of extrapolation techniques in accelerating first-order methods, we study how widely used extrapolation techniques such as those in \cite{BeTe09,AuTe06,LaLuMo11,Nes13} can be incorporated to possibly accelerate the iteratively reweighted $\ell_1$ algorithm. We consider three versions of such algorithms. For each version, we exhibit an explicitly checkable condition on the extrapolation parameters so that the sequence generated provably clusters at a stationary point of the optimization problem. We also investigate global convergence under additional Kurdyka-{\L}ojasiewicz assumptions on certain potential functions. Our numerical experiments show that our algorithms usually outperform the general iterative shrinkage and thresholding algorithm in \cite{GoZhLuHuYe13} and an adaptation of the iteratively reweighted $\ell_1$ algorithm in \cite[Algorithm~7]{Lu2014} with nonmonotone line-search for solving random instances of log penalty regularized least squares problems in terms of both CPU time and solution quality.
\end{abstract}



\section{Introduction}
Originally proposed for unveiling underlying sparse solutions of an underdetermined system, the iteratively reweighted $\ell_1$ algorithm has been widely studied for minimizing optimization models that attempt to induce sparsity in the solutions; see, for example, \cite{CaWaBo08,ChYi08,FoLa09,Lu2014,WiNa10,ZhLi12}.
Here, we consider this algorithm for solving the following class of optimization problem
\begin{equation}\label{primal}
v:=\min\limits_{x\in \mathbb{R}^{n}} F(x):=f(x)+\delta_C(x)+\Phi(|x|),
\end{equation}
where $f:\mathbb{R}^{n}\rightarrow\mathbb{R}$ is a smooth convex function with a Lipschitz continuous gradient whose Lipschitz modulus is $L$, $C$ is a nonempty closed convex set, and $x\mapsto \Phi(|x|)$ is a sparsity inducing function: specifically, we assume that $\Phi(y) = \sum_{i=1}^n\phi(y_i)$, where $\phi:\mathbb{R}_+\rightarrow \mathbb{R}_+$ is a continuous concave function with $\phi(0) = 0$ that is continuously differentiable on $(0,\infty)$, and the limit $\ell:=\lim_{t\downarrow 0}\phi'(t)$ exists. We also assume that $F$ is level-bounded and hence $v>-\infty$ as a consequence. Problem~\eqref{primal} arises naturally in applications such as compressed sensing \cite{CaTa2005,FoRa13} and statistical variable selections \cite{Tib96,FaLi01,ZoTr05,Zha10}, where $f$ is typically a loss function for data fidelity, $C$ is a simple closed convex set (encoding, for example, nonnegativity constraints or box constraints), and $\phi$ can be, for example, the smoothly clipped absolute deviation (SCAD) function \cite{FaLi01}, the minimax concave penalty (MCP) function \cite{Zha10} or the log penalty function \cite{Nik05}.\footnote{Note that when $f$ is the least squares loss function and $\Phi(|\cdot|)$ is the MCP or SCAD function, the function $f(\cdot)+\Phi(|\cdot|)$ is not level-bounded (though it necessarily has a minimizer). However, the level-boundedness of $F$ can still be enforced by picking $C$ to be a huge box, i.e., $C = [-M,M]^n$ for a sufficiently large $M > 0$ so that $C$ intersects $\Argmin_x\{f(x) + \Phi(|x|)\}$. For this choice of $C$, the optimal value of $F$ is the same as that of $f(\cdot)+\Phi(|\cdot|)$.}

When the iteratively reweighted $\ell_1$ algorithm is applied to \eqref{primal}, the sparsity inducing regularizer in the objective is approximated by a weighted $\ell_1$ norm in each iteration, and the resulting subproblem is then approximately solved to produce the next iterate. These subproblems may not have closed form solutions in general due to the function $f$, and a variant of the algorithm was proposed in \cite{Lu2014} that allows majorization of the smooth function $f$ by a quadratic function with a constant Hessian in each iteration to simplify the subproblem. Moreover, a line-search strategy was incorporated for empirical acceleration; see \cite[Algorithm~7]{Lu2014}.

While performing line-search provides an empirical way for accelerating an optimization method, incorporating extrapolation is another classical technique for empirical acceleration. The application of extrapolation techniques has a long history, dating back to Polyak's heavy ball method \cite{Polyak64}. During the past decade, Nesterov's extrapolation techniques \cite{Nes83,Nes03,Nes09,Nes13} have been widely adopted and so-called optimal first-order methods have been developed for convex composite optimization problems; see, for example, \cite{Nes83,AuTe06,BeCaGr11,BeTe09,Tse10,LaLuMo11}. Roughly speaking, these are first-order methods that exhibit a function value convergence rate of $O(1/k^2)$, where $k$ is the number of iterations. Extrapolation techniques have also been applied to the proximal gradient algorithm for some classes of nonconvex problems and good empirical performance has been observed; see, for example, \cite{GhLa16,WeChPo16,XuYi13}.

In view of the success in accelerating first-order methods such as the proximal gradient algorithm empirically via extrapolations, in this paper, we investigate how extrapolation techniques can be suitably incorporated into iteratively reweighted $\ell_1$ algorithms for solving \eqref{primal}. We specifically consider extrapolation techniques motivated from three popular optimal first-order methods: the fast iterative soft-thresholding algorithm (FISTA) \cite{BeTe09,Nes13}, the method by Auslender and Teboulle \cite{AuTe06} and the method by Lan, Lu and Monteiro \cite{LaLuMo11}. We call the corresponding iteratively reweighted $\ell_1$ algorithms with extrapolation IRL$_1e_1$, IRL$_1e_2$ and IRL$_1e_3$, respectively. For each algorithm, we show that the sequence generated clusters at a stationary point of \eqref{primal} under certain condition on the extrapolation parameters. These conditions are satisfied by many choices of extrapolation parameters: for instance, one can pick the parameters as in FISTA with fixed restart \cite{ODoCa15} in IRL$_1e_1$. Furthermore, under some additional assumptions such as the Kurdyka-{\L}ojasiewicz property on some suitable potential functions (see for example, \cite{AtBoReSo10,AtBoSv13}), we show that the whole sequence generated by IRL$_1e_1$ and IRL$_1e_3$ are indeed convergent. We then perform numerical experiments comparing our algorithms (with our proposed choices of extrapolation parameters) against the general iterative shrinkage and thresholding algorithm (GIST) \cite{GoZhLuHuYe13} and an adaptation of the iteratively reweighted $\ell_1$ algorithm \cite[Algorithm~7]{Lu2014} with nonmonotone line-search (IRL$_1{ls}$) for solving log penalty regularized least squares problems on random instances. In our experiments, our iteratively reweighted $\ell_1$ algorithms with extrapolation usually outperform GIST and IRL$_1{ls}$ in both CPU time and solution quality. Moreover, IRL$_1e_1$ and IRL$_1e_3$ usually perform better than IRL$_1e_2$.

The rest of the paper is organized as follows. We present notation and some preliminary results in Section~\ref{preliminaries}. The convergence analyses of IRL$_1e_1$, IRL$_1e_2$ and IRL$_1e_3$ are presented in Sections~\ref{SecFirstExtr}, \ref{sec4} and \ref{sec5} respectively, and our numerical experiments are presented in Section~\ref{sec6}. Finally, some concluding remarks are given in Section~\ref{sec7}.

\section{Notation and preliminaries}\label{preliminaries}
In this paper, we use $\mathbb{R}^n$ to denote the $n$-dimensional Euclidean space with inner product $\left<\cdot,\cdot\right>$ and Euclidean norm $\|\cdot\|$ and use $\|\cdot\|_1$ and $\|\cdot\|_\infty$ to denote the $\ell_1$ and $\ell_\infty$ norms, respectively. We denote the set of vectors whose components are all nonnegative as $\mathbb{R}_+^n$. We also let $a\circ b$ denote the Hadamard (entrywise) product of $a,b\in\mathbb{R}^n$ and let  $|x|$ denote the vector whose $i$th entry is $|x_i|$.

We say that an extended-real-valued function $f: \mathbb{R}^{n}\rightarrow (-\infty,\infty]$ is proper if its domain ${\rm dom} f:=\{x:\; f(x)<\infty\}\neq \emptyset$. A proper function $f$ is said to be closed if it is lower semicontinuous. For a proper function $f$, the regular subdifferential of $f$ at $x\in {\rm dom} f$ is defined by
\[
\hat{\partial} f(x):=\left\{\zeta :\; \liminf\limits_{z\rightarrow x,z\neq x}\frac{f(z)-f(x)-\left<\zeta,z-x\right>}{\|z-x\|}\geqslant 0 \right\}.
\]
The subdifferential of $f$ at $x\in {\rm dom} f$ (which is also called the limiting subdifferential) is defined by
\begin{equation*}
\partial f(x):=\left\{\zeta :\; \exists x^{k}\stackrel{f}{\rightarrow} x, \zeta^{k}\rightarrow \zeta \ {\rm with\ } \zeta^k \in \hat{\partial} f(x^k)\  {\rm for\  each}\  k \right\},
\end{equation*}
where $x^{k}\stackrel{f}{\rightarrow} x$ means both $x^{k}\to x$ and $f(x^{k})\to f(x)$. Moreover, we set $\partial f(x) = \hat \partial f(x) = \emptyset$ for $x\notin {\rm dom}\, f$ by convention, and we write ${\rm dom}\,\partial f:= \{x:\; \partial f(x)\neq \emptyset\}$. It is known that the limiting subdifferential reduces to the classical subdifferential in convex analysis when $f$ is also convex; see \cite[Proposition~8.12]{RoWe09}. Moreover, $\partial f(x) = \{\nabla f(x)\}$ if $f$ is continuously differentiable at $x$ \cite[Exercise~8.8]{RoWe09}. On the other hand, for a locally Lipschitz function $f$, we recall that the Clarke subdifferential of $f$ at a point $x\in \mathbb{R}^n$ is given by
\[
\partial_{\circ}f(x):=\left\{\zeta: \limsup\limits_{y\rightarrow x, t\downarrow 0}\frac{f(y+th)-f(y)}{t}\geqslant\left<\zeta,h\right>\ {\rm for\ all\ } h\in \mathbb{R}^{n}\right\}.
\]
It is known from \cite[Theorem~5.2.22]{BoZh05} that $\partial f(x)\subseteq \partial_{\circ}f(x)$ for a locally Lipschitz function $f$ at any $x$.

For a nonempty closed set $C\subseteq \mathbb{R}^n$ and a vector $a\in \mathbb{R}^n$, we write for notational simplicity the following set
\[
a\circ C :=\left\{a\circ s\ :\ s\in C \right\}.
\]
The indicator function $\delta_C$ of $C$ is defined as
\[
\delta_C(x):=
\begin{cases}
0&  x\in C,\\
\infty & x\notin C.
\end{cases}
\]
The normal cone of $C$ at an $x\in C$ is defined as $N_C(x):=\partial \delta_C(x)$, and the distance from a point $x\in \mathbb{R}^n$ to $C$ is denoted by ${\rm dist}(x,C)$.

We next recall the Kurdyka-{\L}ojasiewicz (KL) property \cite{AtBo09,AtBoReSo10,AtBoSv13,BoDaLe07}. This property is satisfied by a large class of functions such as proper closed semi-algebraic functions; see, for example, \cite{AtBoReSo10,AtBoSv13}. For notational simplicity, we use $\Xi_a$ to denote the set of all concave continuous functions $\varphi: [0,a)\rightarrow \mathbb{R}_+$ that are continuously differentiable on $(0,a)$ with positive derivatives and satisfy $\varphi(0) = 0$.

\begin{definition}[KL property]
  We say that a proper closed function $h$ satisfies the Kurdyka-{\L}ojasiewicz (KL) property at $\hat x\in {\rm dom}\, \partial h$ if there exist an $a\in (0,\infty]$, a neighborhood $V$ of $\hat{x}$ and a $\varphi\in \Xi_a$ such that for any $x\in V$ with $h(\hat{x})<h(x)<h(\hat{x})+a$, it holds that
     \[
     \varphi'(h(x)-h(\hat{x})){\rm dist}(0,\partial h(x))\geqslant 1.
     \]
  If a function $h$ satisfies the KL property at any point in ${\rm dom} \partial h$, we say that it is a KL function.
\end{definition}

The next lemma concerns the uniformized KL property and was proved in \cite[Lemma~6]{BoSaTe14}.
\begin{lemma}[Uniformized KL property]\label{UKL}
Suppose that $h$ is a proper closed function and let $\Gamma$ be a compact set. If $h$ is a constant on $\Gamma$ and satisfies the KL property at each point of $\Gamma$, then there exist $\epsilon, a >0$ and $\varphi\in \Xi_a$ such that
\[
\varphi'(h(x)-h(\tilde{x})) {\rm dist}\left(0,\partial h(x)\right)\geqslant 1
\]
for any $\tilde{x}\in \Gamma$ and any $x$ satisfying $h(\tilde{x})<h(x)<h(\tilde{x})+a$ and ${\rm dist}(x,\Gamma)<\epsilon$.
\end{lemma}

Before ending this section, we prove two auxiliary lemmas that will be used in our convergence analysis. The first lemma concerns properties of the regularizer $\phi$ in the objective function of \eqref{primal}.

\begin{lemma}\label{phiPropositions}
  Let $\phi:\mathbb{R}_+\rightarrow \mathbb{R}_+$ be a continuous concave function with $\phi(0) = 0$ that is continuously differentiable on $(0,\infty)$. Moreover, suppose that $\ell:=\lim_{t\downarrow 0}\phi'(t)$ exists. Then the following statements hold:
  \begin{enumerate}[{\rm (i)}]
    \item $\phi'(t)$ is nonincreasing and nonnegative when $t>0$, and $\ell=\phi'_{+}(0)\geqslant 0$.\footnote{Here and throughout, $\phi'_+(t)$ denotes the right-hand derivative, i.e., $\phi'_+(t):= \lim_{h\downarrow 0}\frac{\phi(t + h) - \phi(t)}{h}$.}
    \item $\partial\phi(|\cdot|)(t)=\phi_{+}'(|t|)\partial|t|$ for all $t\in \mathbb{R}$.
\end{enumerate}

\end{lemma}
\begin{proof}
Note that (i) can be proved directly using the concavity, continuity and nonnegativity of $\phi$. Here, we only prove (ii). Write $g(t)=\phi(|t|)$.
Then $g$ is differentiable at any $t\neq 0$ with $\partial g(t)=\phi'(|t|)\partial|t|$.

We now consider the case $t = 0$. In this case, we note first from the definition of regular subdifferential that
\begin{equation}\label{partialF}
\begin{split}
&\hat{\partial}g(0):=\left\{\mu : \liminf\limits_{y\rightarrow 0, y\neq 0}\frac{g(y)-g(0)-\mu y}{|y|} \geqslant 0\right\}\\
&=\left\{\mu :\min\left\{\liminf\limits_{y\rightarrow 0,  y>0}\frac{g(y)-g(0)-\mu y}{|y|} ,\liminf\limits_{y\rightarrow 0,  y<0}\frac{g(y)-g(0)-\mu y}{|y|} \right\}\geqslant 0\right\}\\
&=\left\{\mu :\min\left\{\phi'_+(0)-\mu,\phi'_+(0)+\mu\right\}\geqslant 0\right\}=[-\phi'_+(0),\phi'_+(0)].
\end{split}
\end{equation}
In addition, from \cite[Theorem~6.2.5]{BoLe06} and the formula of $\nabla g(t)$ when $t\neq 0$, we have
\begin{equation*}
\partial_{\circ}g(0)={\rm conv\ }\left\{\lim\limits_{i}\nabla g(t_i):t_i\rightarrow 0, t_i\neq0\right\}=[-\ell,\ell].
\end{equation*}
Since $\ell = \phi'_+(0)$ according to (i), the above equality together with \eqref{partialF} gives
\[
[-\phi'_+(0),\phi'_+(0)]=\hat{\partial}g(0)\subseteq \partial g(0)\subseteq \partial_{\circ}g(0)=[-\phi'_+(0),\phi'_+(0)],
\]
where the first inclusion follows from the definition of the subdifferentials and the second inclusion follows from \cite[Theorem~5.2.22]{BoZh05}.
\end{proof}

Our second auxiliary lemma has to do with the first-order necessary condition of \eqref{primal}. Recall that a point $\bar x$ is said to satisfy the first-order necessary condition of \eqref{primal} if
\begin{equation}\label{necessary}
0\in \partial\left( f(\cdot)+\delta_C(\cdot)+\Phi(|\cdot|)\right)(\bar{x}).
\end{equation}
Such a point is called a stationary point, and it is known from \cite[Theorem~10.1]{RoWe09} that any local minimizer of \eqref{primal} is a stationary point.
Note that computing the subdifferential in \eqref{necessary} directly from definition can be complicated. In our next lemma, we show that the subdifferential in \eqref{necessary} equals  the sum of $\nabla f(\bar x)$, the normal cone of $C$ at $\bar x$ and the set $\Phi_{+}'(|\bar x|)\circ\partial|\bar x|$; here and throughout this paper, for any $y\in\mathbb{R}^n$, we write
\begin{equation*}
\begin{split}
\Phi'_{+}(|y|)&:=(\phi_{+}'(|y_{1}|),\phi_{+}'(|y_{2}|),\dots,\phi_{+}'(|y_{n}|))\in \mathbb{R}_+^n,\\
\partial |y|&:= \partial|y_1|\times \partial|y_2|\times \cdots \times \partial|y_n|\subset \mathbb{R}^n.
\end{split}
\end{equation*}
Notice that $\Phi'_{+}(|y|)\in \mathbb{R}_+^n$ for all $y\in \mathbb{R}^n$ is a consequence of Lemma~\ref{phiPropositions}(i).

\begin{lemma}\label{subOfPhi}
The objective of \eqref{primal} satisfies the following equation for any $x\in C$:
\begin{equation*}
\partial F(x)=\nabla f(x)+N_C(x)+\Phi_{+}'(|x|)\circ\partial|x|.
\end{equation*}
\end{lemma}
\begin{proof}
Since $\partial \delta_C(x)\neq \emptyset$ at any $x\in C$, by \cite[Corollary~8.11]{RoWe09} and \cite[Proposition~8.12]{RoWe09}, we know that $f$ and $\delta_C$ are regular at any point in $C$.

Let $\tilde \phi:\mathbb{R}\to \mathbb{R}$ be defined so that $\tilde \phi(t) = \phi(t)$ when $t > 0$ and $\tilde \phi(t) = \ell t$ otherwise. Then $\tilde\phi$ is continuously differentiable in view of Lemma~\ref{phiPropositions}(i) and $\phi(|t|)= \max\{\tilde{\phi}(t), \tilde{\phi}(-t)\}$ for all $t\in \mathbb{R}$. Using these, we deduce from \cite[Example~10.24(e)]{RoWe09} that $\phi(|\cdot|)$ is amenable. This together with \cite[Exercise~10.26(a)]{RoWe09} implies that  $\Phi(|\cdot|)$ is amenable. Consequently, $\Phi(|\cdot|)$ is regular, thanks to \cite[Exercise~10.25(a)]{RoWe09}.

Therefore, using \cite[Corallary~10.9]{RoWe09}, we see that
\begin{equation*}
\partial F(x)= \nabla f(x)+N_C(x)+\partial \Phi(|x|)=\nabla f(x)+N_C(x)+\Phi_{+}'(|x|)\circ\partial|x|,
\end{equation*}
where  the second equality follows from \cite[Proposition~10.5]{RoWe09} and Lemma~\ref{phiPropositions}(ii).
\end{proof}

\section{Iteratively reweighted $\ell_1$ algorithm with type-I extrapolation}\label{SecFirstExtr}

In this section, we propose and analyze an iteratively reweighted $\ell_1$ algorithm with an extrapolation technique motivated from FISTA \cite{BeTe09,Nes13}: this technique has been widely studied in both convex and nonconvex settings; see, for example, \cite{BeCaGr11,BeTe09,Tse10,Nes13,WeChPo16}. We call the algorithm based on this extrapolation technique the iteratively reweighted $\ell_1$ algorithm with type-I extrapolation (IRL$_1e_1$). This algorithm is presented in Algorithm~\ref{alg1} below.

\begin{algorithm}
\caption{\bf Iteratively reweighted $\ell_1$ algorithm with type-I extrapolation (IRL$_1e_1$)}\label{alg1}
\begin{algorithmic}
\item[Step 0.] Input an initial point $x^0=x^{-1}\in C$ and $\{\beta_k\}\subset [0,1)$. Set $k=0$.

\item[Step 1.] Set
\begin{equation}\label{chap4:scheme}
\begin{cases}
&s^{k+1}=\Phi_{+}'(|x^{k}|); \\
&y^{k}  =x^{k}+\beta_{k}(x^{k}-x^{k-1});\\
&\displaystyle x^{k+1}=\argmin_{y\in C}\left\{\left<\nabla f(y^{k}), y-y^{k}\right>+\frac{L}{2}\|y-y^{k}\|^{2}+\sum\limits_{i=1}^{n}s_{i}^{k+1}|y_{i}|\right\}.
\end{cases}
\end{equation}

\item[Step 2.] If a termination criterion is not met, set $k=k+1$ and go to Step 1.
\end{algorithmic}
\end{algorithm}

We next present our global convergence analysis. We will first characterize the cluster points of the sequence generated by the algorithm under suitable assumptions on $\{\beta_k\}$, and then show that the whole sequence is convergent under further assumptions. Our analysis makes extensive use of the following auxiliary function, and is similar to the analysis in \cite{WeChPo16}:
 \begin{align}\label{H_1}
H_1(x,y):=f(x)+\delta_C(x)+\Phi(|x|)+\frac{L}{2}\|x-y\|^{2}.
\end{align}
We start by showing that any accumulation point of the sequence $\{x^{k}\}$ generated by IRL$_1e_1$ is a stationary point of \eqref{primal} under the additional assumption that $\sup_k\beta_k<1$. This assumption is general enough to accommodate a wide variety of choices of extrapolation parameters such as those used in FISTA with both fixed and adaptive restart strategies \cite{ODoCa15}. This latter choice of $\{\beta_k\}$ was shown empirically to be highly effective in accelerating the proximal gradient algorithm for convex composite optimization problems \cite{ODoCa15} and the proximal difference-of-convex algorithm for a class of difference-of-convex optimization problems \cite{WeChPo16}.

\begin{theorem}\label{Type1:thm1}
 Suppose that $\sup_{k\geqslant 0}\beta_{k}<1$ and let $\{x^{k}\}$ be the sequence generated by IRL$_1e_1$ for solving \eqref{primal}.  Then the following statements hold:
\begin{enumerate}[{\rm(i)}]
  \item $\{H_1(x^{k},x^{k-1})\}_{k\geqslant0}$ is a nonincreasing convergent sequence. Moreover, there exists a positive constant $D_1$ such that
  \begin{align}\label{1type:dayu}
  H_1(x^{k},x^{k-1})-H_1(x^{k+1},x^{k})\geqslant D_1 \|x^{k}-x^{k-1}\|^{2}.
  \end{align}
  \item The sequence $\{x^k\}$ is bounded and $\lim_{k}\|x^{k+1}-x^{k}\|=0$.
  \item Any accumulation point of $\{x^{k}\}$ is a stationary point of \eqref{primal}.
  \end{enumerate}
\end{theorem}
\begin{proof}
First we prove (i). We write $l_f(x;y) := f(y) + \left<\nabla f(y),x-y\right>$ for notational simplicity. Recall that $\nabla f(x)$ is Lipschitz continuous with modulus $L$. Then, we have
\begin{align*}
&F(x^{k+1})\leqslant l_f(x^{k+1};y^k)+\frac{L}{2}\|x^{k+1} -y^{k}\|^{2}+\Phi(|x^{k+1}|)\\
&\leqslant l_f(x^{k+1};y^k)+\frac{L}{2}\|x^{k+1} -y^{k}\|^{2}+\Phi(|x^{k}|)+\sum_{i=1}^{n}s^{k+1}_{i}(|x_{i}^{k+1}|-|x_i^k|),
\end{align*}
where the second inequality follows from the concavity of $\phi$ and the definition of $s^{k+1}$.
Notice that $x^{k+1}$ is the minimizer of a strongly convex objective function by its definition in \eqref{chap4:scheme}. Using this together with the above inequality, we see further that
\begin{equation}\label{ineq1}
\begin{split}
F(x^{k+1}) &\leqslant l_f(x^{k};y^k)+\frac{L}{2}\|x^{k} -y^{k}\|^{2}+\Phi(|x^{k}|)-\frac{L}{2}\|x^{k+1}-x^{k}\|^{2}\\
&\leqslant f(x^{k})+\frac{L}{2}\|x^{k} -y^{k}\|^{2}+\Phi(|x^{k}|)-\frac{L}{2}\|x^{k+1}-x^{k}\|^{2}\\
                                             &= f(x^{k})+\Phi(|x^{k}|)+\frac{L}{2}\beta_{k}^{2}\|x^{k}-x^{k-1}\|^2-\frac{L}{2}\|x^{k+1}-x^{k}\|^{2},
\end{split}
\end{equation}
where the second inequality follows from the convexity of $f$, while the equality follows from the definition of $y^k$ in \eqref{chap4:scheme}.

Rearranging \eqref{ineq1} and invoking the definition of $H_1$ and the fact that $\sup_k\beta_k < 1$, we have
\begin{equation}\label{decrease}
\begin{split}
&0 \leqslant\frac{L}{2}(1-\sup_k\beta_{k}^{2})\|x^{k}-x^{k-1}\|^{2}\leqslant \frac{L}{2}(1-\beta_{k}^{2})\|x^{k}-x^{k-1}\|^{2}\\
&\leqslant \left[F(x^{k})+\frac{L}{2}\|x^{k}-x^{k-1}\|^{2}\right]-\left[F(x^{k+1})+\frac{L}{2}\|x^{k+1}-x^{k}\|^{2}\right]\\
&= H_1(x^{k},x^{k-1})-H_1(x^{k+1},x^{k}),
\end{split}
\end{equation}
which implies that $\{H_1(x^{k},x^{k-1})\}$ is nonincreasing and \eqref{1type:dayu} holds with $D_1 = \frac{L}{2}(1-\sup_k\beta_{k}^{2}) > 0$. In addition, since $\inf F \geqslant v>-\infty$, we know that $\left\{H_1(x^{k},x^{k-1})\right\}$ is bounded from below. Thus, $\lim_{k}H_1(x^{k},x^{k-1})$ exists and (i) holds.

In addition, we have from \eqref{decrease} and the definition of $H_1$ that
\[
F(x^k)\leqslant H_1(x^{k},x^{k-1})\leqslant H_1(x^{0},x^{-1}) =F(x^{0})< \infty,
\]
Since $F$ is level-bounded, we conclude from this inequality that $\{x^k\}$ is bounded.
Moreover, summing \eqref{decrease} from $k=0$ to $\infty$, we obtain
\begin{equation*}
\frac{L}{2}\sum\limits_{k=0}^{\infty}\left[1-(\sup_k\beta_{k})^{2}\right]\|x^{k}-x^{k-1}\|^{2}\leqslant H_1(x^0,x^{-1})-\lim_{k}H_1(x^{k+1},x^k)<\infty.
\end{equation*}
Since $\sup_k\beta_{k}<1$, we have $\lim_{k}\|x^{k+1}-x^k\| = 0$. This proves (ii).

Now we prove (iii). Let $\tilde{x}$ be an accumulation point of $\{x^{k}\}$ and let $\{x^{k_{j}}\}$ be a subsequence such that $x^{k_{j}}\to \tilde{x}$. Using the first-order optimality condition of the subproblem in \eqref{chap4:scheme}, we have
\[
0\in \nabla f(y^{k_{j}})+N_C(x^{k_{j}+1})+L(x^{k_{j}+1}-y^{k_{j}})+s^{k_{j}+1}\circ\partial|x^{k_{j}+1}|;
\]
here we made use of the subdifferential calculus rules in \cite[Proposition~10.5]{RoWe09} and \cite[Proposition~10.9]{RoWe09}.
Combining this with the definition of $y^k$ in \eqref{chap4:scheme} and rearranging terms, we deduce that
\begin{align}\label{optimal_1}
-L[(x^{k_{j}+1}\!\!-x^{k_{j}})-\beta_{k_{j}}\!(x^{k_{j}}-x^{k_{j}-1})]\!\in\! \nabla f(y^{k_{j}})\!+\!N_C(x^{k_{j}+1})+s^{k_{j}+1}\!\!\circ\partial|x^{k_{j}+1}|.
\end{align}

Next, we claim that
\begin{equation}\label{slimit}
\lim_{j}s^{k_j+1}=\Phi_{+}'(|\tilde{x}|).
\end{equation}
To prove this, we first consider those $i$ corresponding to $\tilde{x}_i \neq 0$. Since $\phi$ is continuously differentiable, we have from the definition of $s^{k_j+1}$ that
$\lim_{j}s^{k_j+1}_i=\lim_{j}\phi_{+}'(|x_{i}^{k_{j}}|)=\phi_{+}'(|\tilde{x}_{i}|) $.
On the other hand, for those $i$ corresponding to $\tilde{x}_i = 0$, we have $\lim_{j}s_i^{k_j+1}=\lim_{j}\phi_{+}'(|x_{i}^{k_{j}}|)=\ell=\phi_{+}'(0)$, thanks to Lemma~\ref{phiPropositions}(i). Thus, \eqref{slimit} holds.

Now, notice that $\Phi'_+(x)\circ \partial |x| \subseteq [-\ell,\ell]^n$ for all $x\in \mathbb{R}^n$, meaning that the set-valued mapping $x\rightrightarrows\Phi'_+(x)\circ\partial|x|$ is bounded. Using this, \cite[Proposition~5.51]{RoWe09}, the closedness of convex subdifferentials, \eqref{slimit} and the fact that $\lim_{k}\|x^{k+1}-x^k\|=0$ from (ii), we see by passing to the limit in \eqref{optimal_1} that
\begin{align*}
0\in\nabla f(\tilde{x})+N_C(\tilde{x})+\Phi'_{+}(|\tilde{x}|)\circ\partial|\tilde x|=\partial F(\tilde{x}),
\end{align*}
where the last equality follows from Lemma~\ref{subOfPhi}. Thus (iii) holds.
\end{proof}

\begin{corollary}\label{constant}
  Suppose that $\sup_{k\geqslant 0}\beta_{k}<1$ and let $\{x^{k}\}$ be the sequence generated by IRL$_1e_1$ for solving \eqref{primal}. Then the set of accumulation points of $\{(x^k,x^{k-1})\}$, denoted by $\Omega_1$, is a nonempty compact subset of $\dom \partial H_1$, and $H_1\equiv \lim_kH_1(x^{k},x^{k-1})$ on $\Omega_1$.
\end{corollary}
\begin{proof}
 From Theorem~\ref{Type1:thm1}(ii), we know that the set of accumulation points of $\{x^k\}$, denoted by $\Lambda_1$, is nonempty and compact. Moreover, since $\lim_{k}\|x^{k+1}-x^k\|=0$, we see that $\Omega_1 = \{(x,x):\; x\in \Lambda_1\}$, which is clearly nonempty and compact. Furthermore, since $\Lambda_1$ belongs to $\{x:\; 0\in \partial F(x)\}\subseteq \dom \partial F$ according to Theorem~\ref{Type1:thm1}(iii), it is routine to check that $\Omega_1\subset \dom \partial H_1$.

Next, choose any $(\tilde x,\tilde x)\in \Omega_1$ and let $\{x^{k_j}\}$ be a subsequence of $\{x^{k}\}$ with $x^{k_j}\to\tilde{x}$. Then
\begin{equation*}
    H_1(\tilde{x},\tilde{x})=F(\tilde{x}) = \lim_{j}F(x^{k_j})+\frac{L}{2}\|x^{k_j}-x^{k_j-1}\|^2= \lim_{j}H_1(x^{k_j},x^{k_j-1}),
\end{equation*}
where the second equality follows from the continuity of $F$ on $C$ and Theorem \ref{Type1:thm1}(ii). Since $\{H_1(x^{k},x^{k-1})\}$ is convergent thanks to Theorem \ref{Type1:thm1}(i) and $(\tilde x,\tilde x)\in \Omega_1$ is chosen arbitrarily, we obtain that $H_1\equiv \lim_kH_1(x^{k},x^{k-1})$ on $\Omega_1$. This completes the proof.
\end{proof}

Next, we prove under additional assumptions on $H_1$ and $\phi_+'$ that the whole sequence $\{x^k\}$ generated by IRL$_1e_1$ is convergent to a stationary point of \eqref{primal}. We start with an auxiliary lemma.
\begin{lemma}\label{lemma1}
Suppose that $\sup_{k\geqslant 0}\beta_{k}<1$ and that $\phi_{+}'$ is Lipschitz continuous on $[0,\infty)$. Let $\{x^{k}\}$ be the sequence generated by IRL$_1e_1$ for solving \eqref{primal}. Then there exists a positive constant $C_1$ such that for all $k\geqslant1$,
\begin{equation*}
  {\rm dist}\left((0,0),\partial H_1(x^{k},x^{k-1})\right)\leqslant C_1(\|x^{k-1}-x^{k-2}\|+\|x^{k}-x^{k-1}\|).
\end{equation*}
\end{lemma}
\begin{proof}
  First, using the first-order optimality condition of the subproblem in \eqref{chap4:scheme} and the definition of $s_{i}^{k}$,
  there exist a $\xi^{k}\in \partial|x^{k}|$ and a $\zeta^k\in N_C(x^k)$ such that
  \begin{align}\label{eqzero}
  0= \nabla f(y^{k-1})+\zeta^k+\Phi'_{+}(|x^{k-1}|)\circ\xi^{k}+L(x^{k}-y^{k-1})
  \end{align}
  for all $k\geqslant 1$.
  Define $\eta^k:= \nabla f(x^{k})+\zeta^k+\Phi'_{+}(|x^{k}|)\circ\xi^{k} +L(x^{k}-x^{k-1})$.
  Then we have
\begin{equation*}
  \begin{split}
  &\ \ \ \ (\eta^k,-L(x^{k}-x^{k-1}))\\
  &\in
  \begin{pmatrix}
    \nabla f(x^{k})+ N_C(x^k)+\Phi'_{+}(|x^{k}|)\circ\partial|x^{k}| +L(x^{k}-x^{k-1})\\
    \{-L(x^{k}-x^{k-1})\}
  \end{pmatrix}\\
  &=\partial H_1(x^{k},x^{k-1}).
  \end{split}
  \end{equation*}
  where the equality follows from \cite[Exercise~8.8]{RoWe09}, \cite[Proposition~10.5]{RoWe09} and Lemma \ref{subOfPhi}. Consequently, we have for all $k\geqslant 0$ that
  \begin{equation}\label{eqhehehaha}
    {\rm dist}\left((0,0),\partial H_1(x^{k},x^{k-1})\right)\leqslant \sqrt{\|\eta^k\|^{2}+L^{2}\|x^{k}-x^{k-1}\|^{2}}.
  \end{equation}

On the other hand, from the definition of $\eta^k$ and \eqref{eqzero}, we see that
 \begin{equation}\label{eqhahahehe}
  \begin{split}
  &\|\eta^k\|=\left\|\eta^k-\left[\nabla f(y^{k-1})+  \zeta^k+\Phi'_{+}(|x^{k-1}|)\circ\xi^{k}+L(x^{k}-y^{k-1})\right]\right\|\\
  &=\left\|\nabla f(x^{k})-\nabla f(y^{k-1}) -L(x^{k-1}-y^{k-1})+\left[\Phi'_{+}(|x^{k}|)-\Phi'_{+}(|x^{k-1}|)\right]\circ\xi^{k}\right\|\\
  &\leqslant \|\nabla f(x^{k})-\nabla f(y^{k-1})\| + L\|x^{k-1}-y^{k-1}\|+\|\Phi'_{+}(|x^{k}|)-\Phi'_{+}(|x^{k-1}|)\|\\
  &\leqslant \|\nabla f(x^{k})-\nabla f(y^{k-1})\| + L\|x^{k-1}-y^{k-1}\|+\sqrt{\sum_{i=1}^n\rho^2(|x^{k}_{i}|-|x^{k-1}_{i}|)^2}\\
  &\leqslant L\|x^k - y^{k-1}\| + L\|x^{k-1}-y^{k-1}\|+\rho\|x^{k}-x^{k-1} \|\\
  & \leqslant (L+\rho) \|x^{k}-x^{k-1} \| + 2L\|x^{k-1}-x^{k-2}\|,
  \end{split}
  \end{equation}
where the first inequality follows from the elementary inequality $\|a\circ b\|\le \|b\|_\infty\|a\|$ for any $a$, $b\in \mathbb{R}^n$ and the fact that $\|\xi^k\|_\infty\le 1$ since $\xi^{k}\in \partial|x^{k}|$, the second inequality follows from the Lipschitz continuity of $\phi'_+$ (with modulus $\rho$); the third inequality holds because $\nabla f$ is Lipschitz continuous; and we made use of the definition of $y^k$ and the fact that $\{\beta_k\}\subset [0,1)$ for the last inequality.
The desired conclusion now follows immediately from \eqref{eqhehehaha} and \eqref{eqhahahehe}.
\end{proof}

We are now ready to prove convergence of the whole sequence $\{x^k\}$ generated by IRL$_1e_1$ under suitable assumptions. Our proof is similar to standard convergence arguments making use of KL property; see, for example, \cite{AtBoReSo10,AtBoSv13}. We include the proof for completeness.

\begin{theorem}\label{thm2}
Suppose that $\sup_{k\geqslant 0}\beta_{k}<1$ and that $\phi_{+}'$ is Lipschitz continuous on $[0,\infty)$. Suppose in addition that $H_1$ is a KL function. Let $\{x^{k}\}$ be the sequence generated by IRL$_1e_1$ for solving \eqref{primal}.  Then $\sum_{k=1}^{\infty}\|x^{k}-x^{k-1}\|<\infty$ and $\{x^{k}\}$ converges to a stationary point of the problem \eqref{primal}.
\end{theorem}
\begin{proof}
In view of Theorem~\ref{Type1:thm1}(iii), it suffices to show that $\sum_{k=1}^{\infty}\|x^{k}-x^{k-1}\|<\infty$ (which implies convergence of $\{x^k\}$).
To this end, note first from Theorem~\ref{Type1:thm1}(i) that $w_1:= \lim_{k}H_1(x^{k},x^{k-1})$ exists.
If there exists $k'$ such that $H_1(x^{k'},x^{k'-1})=w_1$, then for all $k\geqslant k'$, we must have $H_1(x^{k},x^{k-1})=H_1(x^{k'},x^{k'-1})=w_1$, thanks to the fact that $\{H_1(x^{k},x^{k-1})\}$ is nonincreasing by Theorem~\ref{Type1:thm1}(i). Combining this with \eqref{1type:dayu}, we obtain that $x^{k}=x^{k'}$ when $k\geqslant k'$, i.e. the sequence generated converges finitely and thus the conclusion of this theorem holds in this case. Thus, from now on, we assume that $H_1(x^{k},x^{k-1})>w_1$ for all $k$.

According to Corollary~\ref{constant}, $H_1$ is constant (which equals $w_1$) on the nonempty compact set $\Omega_1\subseteq \dom \partial H$, where $\Omega_1$ is the set of accumulation points of $\{(x^k,x^{k-1})\}$. This together with the assumption that $H_1$ is a KL function and Lemma \ref{UKL} implies that there exist $\epsilon_1$, $\eta_1 >0$ and $\varphi_1\in \Xi_{\eta_1}$ such that
\[
\varphi_1'\left(H_1(x,y)-w_1\right) {\rm dist}\left(0,\partial H_1(x,y)\right)\geqslant 1
\]
for any $(x,y)$ satisfying ${\rm dist}((x,y),\Omega_1)<\epsilon_1$ and $w_1<H_1(x,y)<w_1+\eta_1$.
In addition, since $\{x^k\}$ is bounded according to Theorem~\ref{Type1:thm1}(ii), there exists $k_0$ such that whenever $k\geqslant k_0$, we have
\[
{\rm dist}((x^k,x^{k-1}),\Omega_1)<\epsilon_1.
\]
Furthermore, from the definition of $w_1$ and the assumption that $H_1(x^{k},x^{k-1})>w_1$ for all $k$, we know there exists a $k_1$ such that such that whenever $k\geqslant k_1$, $w_1<H_1(x^k,x^{k-1})<w_1+\eta_1$.
Let $N=\max\left\{k_0,k_1\right\}$. Then for $k>N$, we have
\[
\varphi_1'\left(H_1(x^k,x^{k-1})-w_1\right) {\rm dist}\left(0,\partial H_1(x^k,x^{k-1})\right)\geqslant 1.
\]
Combining this with the concavity of $\varphi_1$ we have
\begin{equation*}
  \begin{split}
   & [\underbrace{\varphi_1\left(H_1(x^k,x^{k-1})-w_1\right)-\varphi_1\left(H_1(x^{k+1},x^{k})-w_1\right)}_{\Delta_k}]\cdot{\rm dist}\left(0,\partial H_1(x^k,x^{k-1})\right)\\
   &\geqslant  \varphi_1'\left(H_1(x^k,x^{k-1})-w_1\right)\cdot {\rm dist}\left(0,\partial H_1(x^k,x^{k-1})\right)\cdot\left(H_1(x^k,x^{k-1})-H_1(x^{k+1},x^{k})\right)\\
   & \geqslant H_1(x^k,x^{k-1})-H_1(x^{k+1},x^{k})\geqslant D_1 \|x^{k}-x^{k-1}\|^{2},
  \end{split}
\end{equation*}
where the last inequality follows from  \eqref{1type:dayu}.
Using Lemma~\ref{lemma1} to upper bound the term ${\rm dist}\left(0,\partial H_1(x^k,x^{k-1})\right)$ in the above relation, we further deduce that
\begin{equation*}
\begin{split}
\|x^{k}-x^{k-1}\|^{2}&\leqslant \frac{4C_1}{D_1}\Delta_k\cdot\frac{1}{4}\left(\|x^{k-1}-x^{k-2}\|+\|x^{k}-x^{k-1}\|\right)\\
&\leqslant \left[\frac{C_1}{D_1}\Delta_k+\frac{1}{4}\left(\|x^{k-1}-x^{k-2}\|+\|x^{k}-x^{k-1}\|\right)\right]^2,
\end{split}
\end{equation*}
where the second inequality follows the relation $4ab\leqslant (a+b)^2$ for $a$, $b\in \mathbb{R}$. Taking square root on both sides of the above inequality and rearranging terms, we have
\begin{equation*}
\frac{1}{2}\|x^{k}-x^{k-1}\|\leqslant\frac{C_1}{D_1}\Delta_k+\frac{1}{4}\left(\|x^{k-1}-x^{k-2}\|-\|x^{k}-x^{k-1}\|\right).
\end{equation*}
Summing this inequality from $k = N+1$ to $\infty$, we obtain that
\begin{equation*}
\frac{1}{2}\sum\limits_{k=N+1}^{\infty}\|x^{k}-x^{k-1}\|\leqslant\frac{C_1}{D_1}\left(\varphi_1\left(H_1(x^{N+1},x^N)-w_1\right)\right)+\frac{1}{4}\left(\|x^{N}-x^{N-1}\|\right)< \infty,
\end{equation*}
which also implies the convergence of $\{x^k\}$. This completes the proof.
\end{proof}

\section{Iteratively reweighted $\ell_1$ algorithm with type-II extrapolation}\label{sec4}

In this section, we propose and analyze another version of iteratively reweighted $\ell_1$ algorithm with an extrapolation technique motivated from the method by Auslender and Teboulle \cite{AuTe06}: this was described as the second APG method in \cite{Tse10} and was shown empirically to be the most efficient optimal first-order method in the numerical experiments of \cite{BeCaGr11}. We call the algorithm based on this extrapolation technique the iteratively reweighted $\ell_1$ algorithm with type-II extrapolation (IRL$_1e_2$). This method is presented as Algorithm~\ref{alg2} below.

\begin{algorithm}
\caption{\bf Iteratively reweighted $\ell_1$ algorithm with type-II extrapolation (IRL$_1e_2$)}\label{alg2}
\begin{algorithmic}
\item[Step 0.] Input initial points $x^{0}, z^0\in  C$ and a sequence $\{\theta_k\}\subset (0,1]$. Set $k = 0$.

\item[Step 1.]Set
\begin{equation}\label{chap:2thExtraScheme}
  \begin{cases}
  &s^{k+1}=\Phi'_{+}(|x^k|);\\
  &y^k=(1-\theta_{k})x^k+\theta_{k}z^k;\\
&\displaystyle z^{k+1}= \argmin\limits_{x\in C}\left\{\left<\nabla f(y^k),x-y^k\right> +\frac{L\theta_{k}}{2}\|x-z^{k}\|^{2}+\sum\limits_{i=1}^n s_{i}^{k+1}|x_{i}|\right\};\\
&x^{k+1}=(1-\theta_{k})x^k+\theta_{k}z^{k+1}.
\end{cases}
\end{equation}

\item[Step 2.] If a termination criterion is not met, set $k = k+1$ and go to Step 1.
\end{algorithmic}
\end{algorithm}

We will show that any accumulation point of the sequence $\{z^k\}$ generated by IRL$_1e_2$ is a stationary point of $F$ under suitable assumptions. Our convergence arguments also make use of $H_1$ defined in \eqref{H_1}, and are inspired by \cite[Appendix~A]{Tse10}. In our analysis below, the parameters $\{\theta_k\}$ in IRL$_1e_2$ have to satisfy \eqref{alpha}. We will demonstrate in Section~\ref{sec6} how such $\{\theta_k\}$ can be chosen in our numerical experiments.

\begin{theorem}\label{Extra2thm1}
  Suppose that the $\{\theta_k\}$ in IRL$_1e_2$ is chosen so that
  \begin{equation}\label{alpha}
  \sup_{k\geqslant 1}\{\theta_{k}^2(1-\theta_{k-1})^2 - \theta_{k-1}^2\} < 0,
  \end{equation}
  and let $\{x^k,y^k,z^k\}$ be the sequences generated by IRL$_1e_2$ for solving \eqref{primal}.
  Then the following statements hold.
\begin{enumerate}[{\rm (i)}]
  \item $\{H_1(x^{k},x^{k-1})\}_{k\geqslant1}$ is a nonincreasing convergent sequence.
  \item It holds that
  \begin{equation}\label{allGoTo0}
\lim_k\max\{\|z^{k+1}-x^{k}\|,\|z^{k+1}-y^k\|,\|z^{k+1}-z^k\|\}=0.
  \end{equation}
  \item The sequence $\{z^k\}$ is bounded.
  \item Any accumulation point of $\{z^{k}\}$ is a stationary point of \eqref{primal}.
\end{enumerate}
\end{theorem}
%
\begin{proof}
  In this proof, we write $l_f(x;y):= f(y)+\left<\nabla f(y),x-y\right>$ for notational simplicity.
  Since $\nabla f$ is Lipschitz continuous with modulus $L > 0$, we have
  \begin{equation}\label{eqhaha-1}
  \begin{split}
    &F(x^{k+1})\leqslant l_f(x^{k+1};y^k)+\frac{L}{2}\|x^{k+1}-y^k\|^{2}+\Phi(|x^{k+1}|)\\
    &=  l_f(x^{k+1};y^k)+\frac{L\theta_{k}^2}{2}\|z^{k+1}-z^k\|^{2}+\Phi(|x^{k+1}|)\\
    &=(1-\theta_{k})l_f(x^k;y^k)+\theta_{k}l_f(z^{k+1};y^k)+\frac{L\theta_{k}^2}{2}\|z^{k+1}-z^k\|^{2}+\Phi(|x^{k+1}|)\\
    &=(1-\theta_{k})l_f(x^k;y^k)+\theta_{k}\left[l_f(z^{k+1};y^k)+\frac{L\theta_{k}}{2}\|z^{k+1}-z^k\|^2+\sum\limits_{i=1}^n s_{i}^{k+1}|z_{i}^{k+1}|\right]\\
    &\ \ \ \ + \sum_{i=1}^n\left[\phi(|x_i^{k+1}|)-\theta_{k} s_{i}^{k+1}|z_{i}^{k+1}|\right]\\
    &\leqslant l_f(x^k;y^k)+\theta_{k}\left[\frac{L\theta_{k}}{2}\|x^k-z^k\|^2+\sum\limits_{i=1}^n s_{i}^{k+1}|x^k_{i}|-\frac{L\theta_{k}}{2}\|x^k-z^{k+1}\|^2\right]\\
    &\ \ \ \ + \sum_{i=1}^n\left[\phi(|x_i^{k+1}|)-\theta_{k} s_{i}^{k+1}|z_{i}^{k+1}|\right],
  \end{split}
  \end{equation}
  where the first equality follows from the definitions of $x^{k+1}$ and $y^k$ in \eqref{chap:2thExtraScheme} so that
  \begin{equation*}
  x^{k+1}-y^k = [(1-\theta_{k})x^k+\theta_{k}z^{k+1}] - [(1-\theta_{k})x^k+\theta_{k}z^k] = \theta_k(z^{k+1}-z^k),
  \end{equation*}
  the second equality follows from the definition of $x^{k+1}$, and the last inequality follows from the definition of $z^{k+1}$ as a minimizer and the strong convexity of the objective of the minimization problem defining $z^{k+1}$.

  From the convexity of $f$ and \eqref{eqhaha-1}, we see further that
  \begin{equation}\label{eqhaha}
  \begin{split}
   F(x^{k+1})&\leqslant f(x^k)+\theta_{k}\left[\frac{L\theta_{k}}{2}\|x^k-z^k\|^2+\sum_{i=1}^n s_{i}^{k+1}|x_{i}^{k}|-\frac{L\theta_{k}}{2}\|x^k-z^{k+1}\|^2\right]\\
   &\ \ \ \ + \sum\limits_{i=1}^n\left[\phi(|x_i^{k+1}|)-\theta_{k} s_{i}^{k+1}|z_{i}^{k+1}|\right],\\
   &\leqslant f(x^k)+\theta_{k}\left[\frac{L\theta_{k}}{2}\|x^k-z^k\|^2-\frac{L\theta_{k}}{2}\|x^k-z^{k+1}\|^2\right]\\
   &\ \ \ \ +\sum\limits_{i=1}^n\left[\phi(|x_i^{k}|)+s_i^{k+1}(|x^{k+1}_{i}|- |x^{k}_{i}|)-\theta_{k} s_{i}^{k+1}|z_{i}^{k+1}|+\theta_k s_{i}^{k+1}|x_{i}^{k}|\right],
  \end{split}
  \end{equation}
  where the second inequality follows from the fact that $s^{k+1}=\Phi_+'(|x^k|)$ and the concavity of $\phi$.

  Next, observe from the last relation in \eqref{chap:2thExtraScheme} that for each $i=1,\ldots,n$,
  \begin{equation*}
  \begin{split}
    &|x^{k+1}_i| = |(1-\theta_k)x^k_i + \theta_kz^{k+1}_i|, \\
    \Longrightarrow \ \ \ &|x^{k+1}_i| \leqslant (1-\theta_k)|x^k_i| + \theta_k|z^{k+1}_i|,\\
    \Longrightarrow \ \ \ &|x^{k+1}_i| - |x_i^k| - \theta_k|z_i^{k+1}| + \theta_k|x_i^k|\leqslant 0.
  \end{split}
  \end{equation*}
  Combining this with \eqref{eqhaha} and the nonnegativity of $s_i^k$, we obtain further that for all $k\geqslant1$ that
  \begin{equation}\label{eqhehe}
  \begin{split}
   F(x^{k+1})&\leqslant f(x^k)+\Phi(|x^{k}|)+\theta_{k}\left[\frac{L\theta_{k}}{2}\|x^k-z^k\|^2-\frac{L\theta_{k}}{2}\|x^k-z^{k+1}\|^2\right] \\
   &=    F(x^k)+ \frac{L\theta^2_{k}(1-\theta_{k-1})^2}{2}\|x^{k-1}-z^k\|^2-\frac{L\theta^2_{k}}{2}\|x^k-z^{k+1}\|^2,
  \end{split}
  \end{equation}
  where the last equality follows from the last relation in \eqref{chap:2thExtraScheme}.

  Observe that for $k\geqslant0$, $x^k\in C$ and $\theta_k(z^{k+1}-x^k) = x^{k+1}-x^k$. Using these and the definition of $H_1$, we obtain from \eqref{eqhehe} that for $k\geqslant1$,
  \begin{equation}\label{F:decreasing}
  \begin{split}
     H_1(x^{k+1},x^{k})- H_1(x^k,x^{k-1}) &\leqslant \left(\frac{L\theta_{k}^2(1-\theta_{k-1})^2}{2}-\frac{L\theta^2_{k-1}}{2}\right)\|x^{k-1}-z^k\|^2\\
      &\leqslant -A_1\|x^{k-1}-z^k\|^2,
  \end{split}
  \end{equation}
  where $A_1 := \frac{L}2\inf_k\{\theta^2_{k-1}-\theta_{k}^2(1-\theta_{k-1})^2\}$, which is positive thanks to \eqref{alpha}. Thus, $\{H_1(x^{k},x^{k-1})\}_{k\geqslant1}$ is nonincreasing.

  In addition, since $x^k\in C$, we have
  \[
  v\le F(x^k)\leqslant F(x^k)+\frac{L}{2}\|x^k-x^{k-1}\|^2=H_1(x^k,x^{k-1}),
  \]
  showing that $\{H_1(x^{k},x^{k-1})\}$ is bounded from below. Thus, $\{H_1(x^{k},x^{k-1})\}_{k\geqslant1}$ is convergent. This proves (i).

  Next, summing \eqref{F:decreasing} from $k=1$ to $\infty$, we have
  \[
  A_1\sum_{k=1}^\infty \|x^{k-1}-z^k\|^2  \leqslant H_1(x^{1},x^{0})-\lim_{k}H_1(x^{k+1},x^{k}) < \infty.
  \]
  Therefore, $ \lim_k\|x^k-z^{k+1}\|= 0$, which further implies that

 \begin{equation}\label{equalities}
 \begin{split}
 \lim_{k} x^{k+1}-x^k&=          \lim_{k} \theta_{k}(z^{k+1}-x^k)=0;\\
 \lim_{k}x^{k+1}-z^{k+1}&=       \lim_{k} (1-\theta_{k})(x^k-z^{k+1})=0;\\
 \lim_{k} x^{k+1}-y^{k+1}&=      \lim_{k} \theta_{k+1}(x^{k+1}-z^{k+1})=0;
  \end{split}
 \end{equation}
where the first and second equalities are due to the last relation in \eqref{chap:2thExtraScheme} and the third equality is due to the second relation in \eqref{chap:2thExtraScheme}. Then we have
\begin{align*}
\begin{split}
  &\lim_{k} z^{k+1}-z^{k}=\lim_{k}\ (z^{k+1}- x^k) + (x^k - z^k) = 0,\\
  &\lim_{k} z^{k+1}-y^k = \lim_{k}\ (z^{k+1}-x^k) + (x^k - y^k) = 0.
\end{split}
\end{align*}
This proves (ii).

We now prove (iii). Notice from \eqref{F:decreasing} and the definition of $H_1$ that
\begin{equation*}
F(x^k)\leqslant H_1(x^{k},x^{k-1})\leqslant H_1(x^{1},x^{0}) = F(x^{1})+\frac{L}{2}\|x^1-x^0\|^2< \infty.
\end{equation*}
Since $F$ is level-bounded, we conclude from this inequality that $\{x^k\}$ is bounded. In view of this and the second equality in \eqref{equalities}, we conclude that $\{z^k\}$ is also bounded, i.e., (iii) holds.

Now we prove (iv).  Let $z^*$ be an accumulation point of $\{z^{k}\}$ and let $\{z^{k_j}\}$ be a subsequence such that $z^{k_j}\to z^*$.  Clearly $z^*\in C$. From \eqref{allGoTo0}, we know that
\begin{equation}\label{allGoTo02}
z^{k_j-1}\to z^*,\ \ y^{k_j-1}\to z^*,\ \ x^{k_j-1}\to z^*.
\end{equation}

Using the definition of $z^{k}$ as the minimizer of the optimization problem in \eqref{chap:2thExtraScheme} and the subdifferential calculus rules in \cite[Proposition~10.5]{RoWe09} and \cite[Proposition~10.9]{RoWe09}, we have
\begin{equation}\label{stationOfSub2}
0\in \nabla f(y^{k_{j}-1})+N_C(z^{k_{j}})+\theta_{k_{j}-1}L(z^{k_{j}}-z^{k_{j}-1})+s^{k_j}\circ\partial|z^{k_{j}}|.
\end{equation}

Next we show that
\begin{equation}\label{slimit2}
\lim\limits_{j}s^{k_j}=\Phi_{+}'(|z^*|).
\end{equation}
We consider two cases. First, for those $i$ satisfying $z_i^*\neq 0$, we have from the definition of $s^k$ and \eqref{allGoTo02} that
$\lim_{j}s^{k_j}_i=\phi_{+}'(|z_{i}^{*}|) $. On the other hand, for those $i$ corresponding to $z_i^*= 0$, we have by the definition of $s^k$ that $\lim_{j}s^{k_j}_i=\lim_j\phi'_+(|x_i^{k_j-1}|) = \ell = \phi_{+}'(0)$, thanks to Lemma~\ref{phiPropositions}(i). Therefore, $\lim_{j}s^{k_j}=\Phi_{+}'(|z^*|)$.

Now, notice that the set-valued mapping $x\rightrightarrows\Phi'_+(x)\circ\partial|x|$ is bounded because $\Phi'_+(x)\circ \partial |x| \subseteq [-\ell,\ell]^n$ for all $x\in \mathbb{R}^n$. Using this, \cite[Proposition~5.51]{RoWe09}, the closedness of convex subdifferentials, \eqref{slimit2} and \eqref{allGoTo02}, we see by passing to the limit in \eqref{stationOfSub2} that
\[
0\in \nabla f(z^*)+N_C(z^*)+\Phi_{+}'(|z^*|)\circ\partial|z^*|=\partial F(z^*),
\]
where the equality follows from Lemma \ref{subOfPhi}.
This completes the proof.
\end{proof}

\section{Iteratively reweighted $\ell_1$ algorithm with type-III extrapolation}\label{sec5}

In this section, we propose and analyze yet another version of iteratively reweighted $\ell_1$ algorithm with an extrapolation technique motivated from the method by Lan, Lu and Monteiro \cite{LaLuMo11}: this was stated as algorithm LLM in \cite{BeCaGr11} and was the first of its kinds whose complexity has been established in some nonconvex settings \cite{DrPa16,GhLa16}. We refer to the algorithm based on this extrapolation technique as the iteratively reweighted $\ell_1$ algorithm with type-III extrapolation (IRL$_1e_3$). The method is presented as Algorithm~\ref{alg3} below.

\begin{algorithm}
\caption{\bf Iteratively reweighted $\ell_1$ algorithm with type-III extrapolation (IRL$_1e_3$)}\label{alg3}
\begin{algorithmic}
\item[Step 0.] Input initial points $x^{0}, z^0\in  C$ and a sequence $\{\theta_k\}\subset (0,1]$. Set $k = 0$.

\item[Step 1.]Set
\begin{equation}\label{chap:3rdExtraScheme}
  \begin{cases}
 & s^{k+1}=\Phi'_{+}(|x^k|);\\
  &y^k=(1-\theta_{k})x^k+\theta_{k}z^k;\\
&\displaystyle z^{k+1}= \argmin_{x\in C}\left\{\left<\nabla f(y^k),x-y^k\right> +\frac{L\theta_{k}}{2}\|x-z^{k}\|^{2}+\sum\limits_{i=1}^n s_{i}^{k+1}|x_{i}|\right\};\\
&\displaystyle x^{k+1}= \argmin_{y\in C}\left\{\left<\nabla f(y^{k}), y-y^{k}\right>+\frac{L}{2}\|y-y^{k}\|^{2}+\sum\limits_{i=1}^{n}s_{i}^{k+1}|y_{i}|\right\}.
\end{cases}
\end{equation}

\item[Step 2.] If a termination criterion is not met, set $k = k+1$ and go to Step 1.
\end{algorithmic}
\end{algorithm}

Now we present convergence analysis for this algorithm. Our analysis is inspired by \cite[Section~8]{DrPa16} and relies heavily on the following auxiliary function:
\begin{equation*}
  H_3(x,y,w) = f(x)+\delta_C(x)+\Phi(|x|)+\frac{L}{2}\|w-y\|^2 +\frac{L}{2}\|w-x\|^2.
\end{equation*}
We start by characterizing the accumulation points of the sequence generated by IRL$_1e_3$ under suitable assumptions on $\{\theta_k\}$, and establish the convergence of the whole sequence under additional assumptions. In particular, we require $\{\theta_k\}$ to be chosen so that \eqref{assumptionOnTheta} holds: we will demonstrate how such $\{\theta_k\}$ can be chosen to satisfy this condition in our numerical experiments in Section~\ref{sec6}.

\begin{theorem}\label{Extra3thm1}
  Suppose that $\{\theta_k\}$ in IRL$_1e_3$ is chosen so that for some $\gamma \in (0,1)$,
  \begin{align}\label{assumptionOnTheta}
  \sup_{k\geqslant 1}\max\left\{\frac{\theta_k^2(1-\theta_{k-1})^2}{\gamma}-\theta_{k-1}^2,\frac{\theta_k^2}{1-\gamma}-1\right\} < 0.
  \end{align}
  Let $\{x^k,y^k,z^k\}$ be the sequences generated by IRL$_1e_3$ for solving \eqref{primal} and define $w^{k+1} := (1-\theta_k)x^k + \theta_k z^{k+1}$ for $k\geqslant0$.
  Then the following statements hold:
  \begin{enumerate}
  \item[{\rm(i)}]   $\{H_3(x^{k},x^{k-1},w^k)\}_{k\geqslant1}$ is a nonincreasing convergent sequence. Moreover, there exists a positive constant $D_3$ such that
  \begin{align}\label{dayu33}
  H_3(x^{k}\!,x^{k-1}\!,w^k)\!-\!H_3(x^{k+1}\!,x^{k}\!,w^{k+1})\!\geqslant\! D_3 (\|x^{k-1}\!-z^k\|^2\!+\!\|w^k-x^k\|^2)\!.
  \end{align}
  \item[{\rm(ii)}]  The sequence $\{x^k\}$ is bounded.
  \item[{\rm(iii)}] It holds that
  \begin{equation}\label{equalities33}
  \lim_{k}\max\{\|x^{k-1}-z^k\|,\|w^k-x^k\|,\|x^k-x^{k-1}\|,\|x^k-y^{k-1}\|\}=0.
  \end{equation}
  \item[{\rm(iv)}]  Any accumulation point of $\{x^{k}\}$ is a stationary point of \eqref{primal}.
  \end{enumerate}
\end{theorem}
\begin{proof}
In this proof, we write $l_f(x;y):=f(y)+\left<\nabla f(y),x-y\right> $ for notational simplicity.
Since $\nabla f$ is Lipschitz, we have
\begin{equation*}
  \begin{split}
    &F(x^{k+1})\leqslant l_f(x^{k+1};y^k)+\frac{L}{2}\|x^{k+1}-y^k\|^2+\Phi(|x^{k+1}|)\\
    &=l_f(x^{k+1};y^k)+\frac{L}{2}\|x^{k+1}-y^k\|^2+\sum\limits_{i=1}^n s^{k+1}_i|x^{k+1}_i|+\Phi(|x^{k+1}|)-\sum\limits_{i=1}^n s^{k+1}_i|x^{k+1}_i|\\
    &\leqslant l_f(w^{k+1};y^k)+\frac{L}{2}\|w^{k+1}-y^k\|^2+\sum\limits_{i=1}^n s^{k+1}_i|w^{k+1}_i|+\Phi(|x^{k+1}|)-\sum\limits_{i=1}^n s^{k+1}_i|x^{k+1}_i|\\
    & \ \ \ -\frac{L}{2}\|w^{k+1}-x^{k+1}\|^2,
  \end{split}
\end{equation*}
where $w^{k+1} = (1-\theta_k)x^k+\theta_k z^{k+1}$, and the second inequality follows from the definition of $x^{k+1}$ as the minimizer of the strongly convex subproblem for the $x$-update. Plugging the definition of $w^{k+1}$ into the first three terms in the last inequality above and invoking the definition of $y^k$, we see that
\begin{equation*}
  \begin{split}
    F(x^{k+1})&\leqslant (1-\theta_k)l_f(x^k;y^k)+\theta_k l_f(z^{k+1};y^k)+\frac{L\theta_k^2}{2}\|z^{k+1}-z^k\|^2+\Phi(|x^{k+1}|)\\&
    \ \ \ +\sum_{i=1}^n s^{k+1}_i|(1-\theta_k)x^k_i+\theta_k z^{k+1}_i|- \sum_{i=1}^ns^{k+1}_i|x^{k+1}_i|-\frac{L}{2}\|w^{k+1}-x^{k+1}\|^2.
  \end{split}
\end{equation*}
Applying the relation $|(1-\theta_k)x^k_i+\theta_k z^{k+1}_i|\le (1-\theta_k)|x^k_i|+\theta_k |z^{k+1}_i|$ to the inequality above and grouping terms, we obtain further that $F(x^{k+1})$ is bounded above by
\begin{equation*}
  \begin{split}
 &(1-\theta_k)l_f(x^k;y^k)+\theta_k \left[l_f(z^{k+1};y^k)+\frac{L\theta_k}{2}\|z^{k+1}-z^k\|^2+\sum\limits_{i=1}^n s^{k+1}_i| z^{k+1}_i|\right]\\
     &+\sum\limits_{i=1}^n (1-\theta_k)s^{k+1}_i|x^k_i|+ \Phi(|x^{k+1}|)- \sum_{i=1}^n s^{k+1}_i|x^{k+1}_i|-\frac{L}{2}\|w^{k+1}-x^{k+1}\|^2\\
     &\leqslant l_f(x^k;y^k)+\theta_k \left[\frac{L\theta_k}{2}\|x^{k}-z^k\|^2+\sum\limits_{i=1}^n s^{k+1}_i| x^{k}_i|\right]-\frac{L\theta_k^2}{2}\|x^k-z^{k+1}\|^2\\
     &\ \ +\sum\limits_{i=1}^n (1-\theta_k)s^{k+1}_i|x^k_i|+\Phi(|x^{k+1}|)- \sum_{i=1}^n s^{k+1}_i|x^{k+1}_i|-\frac{L}{2}\|w^{k+1}-x^{k+1}\|^2\\
     &= l_f(x^k;y^k)+\sum\limits_{i=1}^n\left[ s^{k+1}_i| x^{k}_i|+\phi(|x^{k+1}_i|)- s^{k+1}_i|x^{k+1}_i|\right]+\frac{L\theta_k^2}{2}\|x^{k}-z^k\|^2\\
     &\ \ -\frac{L}{2}\|w^{k+1}-x^{k+1}\|^2-\frac{L\theta_k^2}{2}\|x^k-z^{k+1}\|^2,
  \end{split}
\end{equation*}
where the inequality follows from the definition of $z^{k+1}$ as the minimizer of the strongly convex subproblem for the $z$-update.

Applying the convexity of $f$, the concavity of $\phi$ and the fact that $s^{k+1} = \Phi'_+(|x^k|)$ to the above upper bound, we further have
 \begin{equation}\label{extra3thm_1}
  \begin{split}
F(x^{k+1})
     &\leqslant f(x^k)+\sum\limits_{i=1}^n\left[\phi(|x_i^{k}|)+s_i^{k+1}(|x^{k+1}_{i}|- |x^{k}_{i}|)- s^{k+1}_i|x^{k+1}_i|+s^{k+1}_i|x^{k}_i|\right]\\
     &\ \ \ +\frac{L\theta_k^2}{2}\|x^{k}-z^k\|^2 -\frac{L}{2}\|w^{k+1}-x^{k+1}\|^2-\frac{L\theta_k^2}{2}\|x^k-z^{k+1}\|^2\\
      &= F(x^k)+\frac{L\theta_k^2}{2}\|x^{k}-z^k\|^2 -\frac{L}{2}\|w^{k+1}-x^{k+1}\|^2-\frac{L\theta_k^2}{2}\|x^k-z^{k+1}\|^2.
  \end{split}
\end{equation}

Now, observe that $w^k-x^k = (1-\theta_{k-1})x^{k-1}+\theta_{k-1} z^{k}-x^k=x^{k-1}-x^k+\theta_{k-1}(z^k-x^{k-1})$ for any $k\geqslant1$, we thus have
\[
x^k-z^k = x^k-x^{k-1}+x^{k-1}-z^k = (1-\theta_{k-1})(x^{k-1}-z^k)-(w^k-x^k).
\]
Using this and the inequality that $(a+b)^2\leqslant \frac{a^2}{\gamma}+\frac{b^2}{1-\gamma} $, where $\gamma\in (0,1)$ is as in \eqref{assumptionOnTheta}, we deduce further from \eqref{extra3thm_1} that $F(x^{k+1})$ is bounded above by
 \begin{equation*}
  \begin{split}
    &F(x^k) +\frac{L\theta_k^2(1-\theta_{k-1})^2}{2\gamma}\|x^{k-1}-z^k\|^2 +\frac{L\theta_k^2}{2(1-\gamma)}\|w^k-x^k\|^2\\
    &-\frac{L}{2}\|w^{k+1}-x^{k+1}\|^2-\frac{L\theta_k^2}{2}\|x^k-z^{k+1}\|^2\\
& = F(x^k)+\frac{L}2\!\left(\theta_{k-1}^2\|x^{k-1}-z^k\|^2\! +\|w^k-x^k\|^2\!-\|w^{k+1}-x^{k+1}\|^2\!-\theta_k^2\|x^k-z^{k+1}\|^2\right)\\
&\ \ +\left(\frac{\theta_k^2(1-\theta_{k-1})^2}{\gamma}-\theta_{k-1}^2\right)\frac{L}{2}\|x^{k-1}-z^k\|^2+\left(\frac{\theta_k^2}{1-\gamma}-1\right)\frac{L}{2}\|w^k-x^k\|^2.
  \end{split}
\end{equation*}
Using the assumptions on $\theta_k$, we then obtain the following estimate:
 \begin{equation*}
  \begin{split}
  F(x^{k+1})&\leqslant F(x^k) -A_2(\|x^{k-1}-z^k\|^2 + \|w^k-x^k\|^2)\\
  & \ \ \ +\! \frac{L}2\!\left(\theta_{k-1}^2\|x^{k-1}-z^k\|^2\! +\!\|w^k-x^k\|^2\!-\!\|w^{k+1}-x^{k+1}\|^2\!-\theta_k^2\|x^k-z^{k+1}\|^2\right)\\
  &= F(x^k) -A_2(\|x^{k-1}-z^k\|^2 + \|w^k-x^k\|^2)\\
  & \ \ \ +\! \frac{L}2\!\left(\|w^k - x^{k-1}\|^2\! +\!\|w^k-x^k\|^2\!-\!\|w^{k+1}-x^{k+1}\|^2\!-\|w^{k+1}-x^k\|^2\right),
  \end{split}
\end{equation*}
where $A_2 = \frac{L}2\inf_k\min\left\{\theta_{k-1}^2-\frac{\theta_k^2(1-\theta_{k-1})^2}{\gamma},1-\frac{\theta_k^2}{1-\gamma}\right\}$, which is positive according to \eqref{assumptionOnTheta}, and the equality follows from the definition of $w^k$ so that $w^{k} - x^{k-1} = \theta_{k-1}(z^k - x^{k-1})$.
Rearranging terms in the above inequality and invoking the definition of $H_3$, we have for all $k\geqslant1$ that
\begin{align}\label{decrease3}
A_2(\|x^{k-1}-z^k\|^2+\|w^k-x^k\|^2)\leqslant H_3(x^k,x^{k-1},w^k)-H_3(x^{k+1},x^{k},w^{k+1}),
\end{align}
which means that $\{H_3(x^{k},x^{k-1},w^k)\}_{k\geqslant 1}$ is nonincreasing. In addition, it is not hard to see that $\{H_3(x^{k},x^{k-1},w^k)\}$ is bounded from below. Thus, the sequence $\{H_3(x^{k},x^{k-1},w^k)\}$ is convergent. This proves (i).

Next, we have from \eqref{decrease3} that for any $k\geqslant1$ that
\[
F(x^k)\leqslant H_3(x^{k},x^{k-1},w^k)\leqslant H_3(x^{1},x^{0},w^1) < \infty,
\]
Since $F$ is level-bounded, we conclude from this inequality that $\{x^k\}$ is bounded and therefore (ii) holds.

We now prove (iii). Summing \eqref{decrease3} from $k=1$ to $\infty$, we obtain
\begin{equation*}
A_2\sum_{k=1}^{\infty}(\|x^{k-1}-z^k\|^2+\|w^k-x^k\|^2)\leqslant H_3(x^1,x^{0},w^1)-\lim_{k}H_3(x^{k+1},x^{k},w^{k+1})< \infty.
\end{equation*}
Thus, we have
\begin{equation}\label{equationsPre3}
\lim_{k}\|x^{k-1}-z^k\|=\lim_k\|w^k-x^k\|=0.
\end{equation}
Combining these relations with the definition of $w^k$, we have
\begin{equation}\label{equations3_4-1}
  w^k-x^{k-1}=\theta_{k-1}(z^k-x^{k-1})\rightarrow 0.
\end{equation}
Combining this with  \eqref{equationsPre3}, we see further that
\begin{equation}\label{equations3_4}
 x^k-x^{k-1}=x^k-w^k+w^k-x^{k-1}=(x^k-w^k)+\theta_{k-1}(z^k-x^{k-1})\rightarrow 0.
\end{equation}
Combining this with the definition of $y^k$ and \eqref{equationsPre3}, we obtain
\begin{equation}\label{equations3_5}
\begin{split}
 &y^k-x^k=\theta_k(z^k-x^k)=\theta_k(z^k-x^{k-1})+\theta_k(x^{k-1}-x^k)\\
 &=\theta_k(z^k-x^{k-1})+\theta_k\left[(w^k-x^k)+\theta_{k-1}(x^{k-1}-z^k)\right]\\
 &=(\theta_k-\theta_k\theta_{k-1})(z^k-x^{k-1})+\theta_k(w^k-x^k)\rightarrow 0.
 \end{split}
\end{equation}
Using this together with \eqref{equations3_4} and \eqref{equationsPre3}, we deduce that
\begin{equation}\label{equations3_6}
\begin{split}
 & \ \ w^k-y^{k-1}=w^k-x^k+(x^k-x^{k-1})+(x^{k-1}-y^{k-1})\\
 & =\theta_{k-1}(z^k\!-\!x^{k-1})\!+\!(\theta_{k-1}\theta_{k-2}\!-\!\theta_{k-1})(z^{k-1}\!-\!x^{k-2})+\theta_{k-1}(x^{k-1}-w^{k-1})\to 0.
 \end{split}
\end{equation}
Finally, using \eqref{equations3_6}, we have
\begin{equation*}
\begin{split}
 &x^k-y^{k-1}=x^k-w^{k}+w^{k}-y^{k-1}\\
 &\!=(x^k-w^k)\!+\!\theta_{k-1}(z^k\!-\!x^{k-1})\!+\!(\theta_{k-1}\theta_{k-2}\!-\!\theta_{k-1})(z^{k-1}\!-\!x^{k-2})\!+\!\theta_{k-1}(x^{k-1}-w^{k-1}),
 \end{split}
\end{equation*}
which also goes to zero thanks to \eqref{equationsPre3}. This proves (iii).

Finally, we prove (iv). Let $\tilde{x}$ be an accumulation point of $\{x^{k}\}$ and let $\{x^{k_{j}}\}$ be a subsequence such that $x^{k_{j}}\to \tilde{x}$. From the first-order optimality condition of the second subproblem in \eqref{chap:3rdExtraScheme} and subdifferential calculus rules in \cite[Proposition~10.5]{RoWe09} and \cite[Proposition~10.9]{RoWe09}, we have
\begin{align}\label{converge33}
0\in \nabla f(y^{k_{j}})+N_C(x^{k_{j}+1})+L(x^{k_{j}+1}-y^{k_{j}})+s^{k_{j}+1}\circ\partial|x^{k_{j}+1}|.
\end{align}
Using the same arguments as in the proof of Theorem \ref{Type1:thm1}(iv), we have $\lim_{j}s^{k_j+1}=\Phi_{+}'(|\tilde{x}|)$.

Now, observe that $\Phi'_+(x)\circ \partial |x| \subseteq [-\ell,\ell]^n$ so that the set-valued mapping $x\rightrightarrows\Phi'_+(x)\circ\partial|x|$ is bounded. Using these, \eqref{equalities33}, the closedness of convex subdifferentials and \cite[Proposition~5.51]{RoWe09}, passing to the limit as $j$ goes to $\infty$ in \eqref{converge33}, we have
\begin{align*}
0 \in\nabla f(\tilde{x})+N_C(\tilde{x})+\Phi'_{+}(|\tilde{x}|)\circ\partial|\tilde x|=\partial F(\tilde{x}),
\end{align*}
where the last equality follows from Lemma \ref{subOfPhi}. Thus (iv) holds and this completes the proof.
\end{proof}

\begin{corollary}\label{constant33}
  Suppose that the $\{\theta_k\}$ in IRL$_1e_3$ is chosen so that \eqref{assumptionOnTheta} holds. Let $\{x^k,z^k\}$ be the sequences generated by IRL$_1e_3$ for solving \eqref{primal} and define $w^{k+1} := (1-\theta_k)x^k + \theta_k z^{k+1}$ for $k\geqslant0$. Then the set of accumulation points of $\{(x^{k},x^{k-1},w^{k})\}$, denoted by $\Omega_3$, is a nonempty compact subset of $\dom \partial H_3$. Moreover, it holds that $H_3\equiv \lim_k H_3(x^k,x^{k-1},w^k)$ on $\Omega_3$.
\end{corollary}
\begin{proof}
First, we see from Theorem~\ref{Extra3thm1}(ii) that the set of accumulation points of $\{x^k\}$, denoted by $\Lambda_3$, is nonempty and compact. Moreover, in view of Theorem~\ref{Extra3thm1}(iii), we deduce that $\Omega_3 = \{(x,x,x):\; x\in \Lambda_3\}$, which is clearly nonempty and compact. Finally, since $\Lambda_3\subseteq\{x:\; 0\in \partial F(x)\}\subseteq \dom \partial F$ according to Theorem~\ref{Extra3thm1}(iv), it is routine to check that $\Omega_3\subset \dom \partial H_3$ as required.

Now, fix any $(\tilde x,\tilde x,\tilde x)\in \Omega_3$ and let $\{x^{k_j}\}$ be a subsequence of $\{x^{k}\}$ with $x^{k_j}\to\tilde{x}$. Then
\begin{equation*}
  \begin{split}
    H_3(\tilde{x},\tilde{x},\tilde{x}) &= F(\tilde x)=\lim_{j}F(x^{k_j})+\frac{L}{2}\|w^{k_j}-x^{k_j}\|^2 + \frac{L}{2}\|w^{k_j}-x^{k_j-1}\|^2\\
    & = \lim_{j}H_3(x^{k_j},x^{k_j-1},w^{k_j}),
  \end{split}
\end{equation*}
where the second equality follows from the continuity of $F$ on $C$ and Theorem~\ref{Extra3thm1}(iii). Since $\{H_3(x^{k},x^{k-1},w^k)\}$ is convergent according to Theorem~\ref{Extra3thm1}(i) and $(\tilde x,\tilde x,\tilde x)\in \Omega_3$ is chosen arbitrarily, we conclude that $H_3\equiv \lim_kH_3(x^{k},x^{k-1},w^k)$ on $\Omega_3$. This completes the proof.
\end{proof}

Next, we show under some assumptions on $H_3$ and $\phi_+'$ that the sequence $\{x^k\}$ generated by IRL$_1e_3$ converges to a stationary point of \eqref{primal}. We first prove the following auxiliary lemma.
\begin{lemma}\label{lemma1Of33}
Suppose that the $\{\theta_k\}$ in IRL$_1e_3$ is chosen so that \eqref{assumptionOnTheta} holds and that $\phi_{+}'$ is Lipschitz continuous. Let $\{x^k,z^k\}$ be the sequences generated by IRL$_1e_3$ for solving \eqref{primal} and define $w^{k+1} := (1-\theta_k)x^k + \theta_k z^{k+1}$ for $k\geqslant0$. Then there exists a positive constant $C_3$ such that for all $k\geqslant 2$,
        \begin{equation*}
        \begin{split}
        &{\rm dist}\left((0,0,0),\partial H_3(x^{k},x^{k-1},w^k)\right)\\
        &\leqslant C_3(\|x^k-w^k\|+\|z^k-x^{k-1}\|+\|x^{k-1}-w^{k-1}\|+\|z^{k-1}-x^{k-2}\|).
        \end{split}
        \end{equation*}

\end{lemma}
\begin{proof}
  First, using the optimality condition of the $x$-update in \eqref{chap:3rdExtraScheme} and the definition of $s^{k}$,
  there exist a $\xi^{k}\in \partial|x^{k}|$ and a $\zeta^k\in N_C(x^k)$ such that for all $k\geqslant 2$,
  \begin{align}\label{eqzero33}
  0= \nabla f(y^{k-1})+\zeta^k+\Phi'_{+}(|x^{k-1}|)\circ\xi^{k}+L(x^{k}-y^{k-1}).
  \end{align}
  Define $\eta^k:= \nabla f(x^{k})+\zeta^k+\Phi'_{+}(|x^{k}|)\circ\xi^{k} +L(x^{k}-w^{k})$.
  Then we have
\begin{equation*}
\begin{split}
 &(\eta^k,-L(w^{k}-x^{k-1}),L(w^k-x^{k-1})+L(w^k-x^k))\\
 &\in \begin{pmatrix}
 \nabla f(x^{k})+ N_C(x^k)+\Phi'_{+}(|x^{k}|)\partial|x^{k}| +L(x^{k}-w^{k}) \\
 \{-L(w^{k}-x^{k-1})\} \\
 \{L(w^k-x^{k-1})+L(w^k-x^k)\}
\end{pmatrix}\\
&=\partial H_3(x^{k},x^{k-1},w^k);
  \end{split}
\end{equation*}
here the equality follows from \cite[Exercise~8.8]{RoWe09}, \cite[Proposition~10.5]{RoWe09} and Lemma~\ref{subOfPhi}. Hence, there exists a $C_0 > 0$ so that for all $k\geqslant 1$,
\begin{equation}\label{distH3}
{\rm dist}\left((0,0,0),\partial H_3(x^{k},x^{k-1},w^k)\right) \leqslant C_0(\|\eta^k\| + \|w^{k}-x^{k-1}\| + \|w^{k}-x^{k}\|).
\end{equation}

Next, from the definition of $\eta^k$ and \eqref{eqzero33}, we see further that
 \begin{equation}\label{eqhahahehe33}
  \begin{split}
  &\|\eta^k\|=\|\eta^k-\left[\nabla f(y^{k-1})+\zeta^k+\Phi'_{+}(|x^{k-1}|)\circ\xi^{k}+L(x^{k}-y^{k-1})\right]\|\\
  &=\|\nabla f(x^{k})-\nabla f(y^{k-1}) +\left[\Phi'_{+}(|x^{k}|)-\Phi'_{+}(|x^{k-1}|)\right]\circ\xi^{k}-L(w^k-y^{k-1})\|\\
  &\leqslant \|\nabla f(x^{k})-\nabla f(y^{k-1})\|+\|\Phi'_{+}(|x^{k}|)-\Phi'_{+}(|x^{k-1}|)\|+L\|w^k-y^{k-1}\|\\
  &\leqslant \|\nabla f(x^{k})-\nabla f(y^{k-1})\|+\sqrt{\sum_{i=1}^n\rho^2(|x^{k}_{i}|-|x^{k-1}_{i}|)^2}+L\|w^k-y^{k-1}\|\\
  &\leqslant L\|x^k - y^{k-1}\|+\rho\|x^{k}-x^{k-1} \|+L\|w^k-y^{k-1}\|\\
  &\leqslant [L+\rho]\|x^k - x^{k-1}\| + L\|x^{k-1}-y^{k-1}\|+L\|w^k-y^{k-1}\|,
  \end{split}
  \end{equation}
where the first inequality follows from the elementary inequality $\|a\circ b\|\leqslant\|b\|_\infty\|a\|$ for any $a$, $b\in \mathbb{R}^n$ and the fact that $\|\xi^k\|_\infty\leqslant1$ since $\xi^{k}\in \partial|x^{k}|$; the second inequality follows from the Lipschitz continuity of $\phi'_+$ (with modulus $\rho$); the third inequality holds because $\nabla f$ is Lipschitz continuous. The desired conclusion now follows from \eqref{distH3}, \eqref{eqhahahehe33} and the relations \eqref{equations3_4-1}, \eqref{equations3_4}, \eqref{equations3_5}, \eqref{equations3_6}, which state that $x^k-x^{k-1}$, $w^k-x^{k-1}$, $x^{k-1}-y^{k-1}$ and $w^k-y^{k-1}$ can be written as linear combinations of $z^k - x^{k-1}$, $x^k - w^k$, $z^{k-1} - x^{k-2}$, $x^{k-1} - w^{k-1}$ with coefficients at most $2$.
\end{proof}

We will now establish the convergence of the whole sequence $\{x^k\}$ generated by IRL$_1e_3$ under some assumptions. Our analysis is similar to standard convergence analysis based on KL property; see, for example, \cite{AtBoReSo10,AtBoSv13}. We include the proof for the convenience of the readers.

\begin{theorem}\label{thm2_extra3}
Suppose that the $\{\theta_k\}$ in IRL$_1e_3$ is chosen so that \eqref{assumptionOnTheta} holds, that $H_3$ is a KL function and that $\phi_{+}'$ is Lipschitz continuous. Let $\{x^{k}\}$ be the sequence generated by IRL$_1e_3$ for solving \eqref{primal}. Then $\sum_{k=1}^{\infty}\|x^{k}-x^{k-1}\|<\infty$ and $\{x^{k}\}$ converges to a stationary point of \eqref{primal}.
\end{theorem}
\begin{proof}
In view of Theorem~\ref{Extra3thm1}(iv), it suffices to prove $\sum_{k=1}^{\infty}\|x^{k}-x^{k-1}\|<\infty$, which implies convergence of $\{x^{k}\}$.
Now, recall from Theorem~\ref{Extra3thm1}(i) that $w_3:= \lim_{k}H_3(x^{k},x^{k-1},w^k)$ exists.
If there exists $k'\ge 1$ such that $H_3(x^{k'},x^{k'-1},w^{k'})=w_3$, then \eqref{dayu33} implies that for any $k\geqslant k'$, $H_3(x^{k},x^{k-1},w^k)=H_3(x^{k'},x^{k'-1},w^{k'})=w_3$. Invoking \eqref{dayu33} again together with \eqref{equations3_4}, we obtain that $x^{k}=x^{k'}$ when $k\geqslant k'$, i.e. the sequence generated converges finitely and hence the conclusion of this theorem holds trivially. In what follows, we consider the case where $H_3(x^{k},x^{k-1},w^k)>w_3$ for all $k$.

Recall from Corollary~\ref{constant33} that $\Omega_3$ is a nonempty compact subset of $\dom \partial H_3$ and $H_3 \equiv w_3$. Since $H_3$ is a KL function, using Lemma \ref{UKL}, there exist $\epsilon_3, \eta_3 >0$ and $\varphi_3\in \Xi_{\eta_3}$ such that
\[
\varphi_3'\left(H_3(x,y,w)-w_3\right) {\rm dist}\left(0,\partial H_3(x,y,w)\right)\geqslant 1
\]
for any $(x,y,w)$ satisfying ${\rm dist}((x,y,w),\Omega_3)<\epsilon_3$ and $w_3<H_3(x,y,w)<w_3+\eta_3$. In addition, recall from Corollary~\ref{constant33} that $\Omega_3$ is the set of accumulation points of $\{(x^k,x^{k-1},w^k)\}$. Since $\{(x^k,x^{k-1},w^k)\}$ is bounded in view of Theorem~\ref{Extra3thm1}(ii) and (iii), there exists $k_0$ such that whenever $k\geqslant k_0$,
\[
{\rm dist}((x^k,x^{k-1},w^k),\Omega_3)<\epsilon_3.
\]
Furthermore, it follows from the definition of $w_3$ that there exists $k_1$ such that whenever $k\geqslant k_1$, $w_3<H_3(x^k,x^{k-1},w^k)<w_3+\eta_3$.
Define $N'=\max\left\{k_0,k_1\right\}$. Then for all $k>N'$, we have
\[
\varphi_3'(\underbrace{H_3(x^k,x^{k-1},w^k)-w_3}_{\chi_k}) \cdot{\rm dist}\left(0,\partial H_3(x^k,x^{k-1},w^k)\right)\geqslant 1.
\]
Using this and the concavity of $\varphi_3$, we have further that
\begin{equation*}
  \begin{split}
   & (\underbrace{\varphi_3(\chi_k)-\varphi_3(\chi_{k+1})}_{\Delta_k})\cdot{\rm dist}\left(0,\partial H_3(x^k,x^{k-1},w^k)\right)\\
   &\geqslant  \varphi_3'(\chi_k)\cdot {\rm dist}\left(0,\partial H_3(x^k,x^{k-1},w^k)\right)\cdot(\chi_k - \chi_{k+1})\\
   & \geqslant H_3(x^k,x^{k-1},w^k)-H_3(x^{k+1},x^{k},w^{k+1})\geqslant D_3\left(\|x^{k-1}-z^k\|^2+\|w^k-x^k\|^2\right),
  \end{split}
\end{equation*}
where the last inequality is from \eqref{dayu33}.
Now, using the relations that $(a+b)^2\leqslant 2(a^2+b^2)$ and $4ab\leqslant (a+b)^2$ for $a$, $b\in \mathbb{R}$, we further deduce from this inequality that
\begin{equation*}
\begin{split}
&\left(\|x^{k-1}-z^k\|+\|w^k-x^k\|\right)^2\leqslant 2\left(\|x^{k-1}-z^k\|^2+\|w^k-x^k\|^2\right)\\
&\leqslant \frac{8C_3}{D_3}\Delta_k \cdot\frac{1}{4C_3}{\rm dist}\left(0,\partial H_3(x^k,x^{k-1},w^k)\right)\\
&\leqslant \frac{8C_3}{D_3}\Delta_k \cdot\frac{1}{4}\left(\|x^k-w^k\|+\|z^k-x^{k-1}\|+\|x^{k-1}-w^{k-1}\|+\|z^{k-1}-x^{k-2}\|\right)\\
&\leqslant \left[\frac{2C_3}{D_3}\Delta_k
+\frac{1}{4}\left(\|x^k-w^k\|+\|z^k-x^{k-1}\|+\|x^{k-1}-w^{k-1}\|+\|z^{k-1}-x^{k-2}\|\right)\right]^2,
\end{split}
\end{equation*}
where the third inequality follows from Lemma~\ref{lemma1Of33}. Taking square root on both sides of the above inequality and rearrange terms, we obtain
\begin{equation*}
\begin{split}
&\frac{1}{2}\left(\|x^{k-1}-z^k\|+\|w^k-x^k\|\right)\\
&\leqslant\frac{2C_3}{D_3}\Delta_k+\frac{1}{4}\left[\|x^{k-1}-w^{k-1}\|+\|z^{k-1}-x^{k-2}\|-\|x^k-w^k\|-\|z^k-x^{k-1}\|\right].
\end{split}
\end{equation*}
Summing this inequality from $k=N'+1$ to $\infty$, using \eqref{equations3_4} and the fact that $H_3(x^{k},x^{k-1})>w_3$ for all $k$,  we obtain that
\begin{equation*}
\begin{split}
&\frac{1}{2}\sum_{k=N'+1}^{\infty}\|x^k-x^{k-1}\|=\frac{1}{2}\sum_{k=N'+1}^{\infty}\|x^k-w^k+\theta_{k-1}(z^k-x^{k-1})\|\\
&\leqslant \frac{1}{2}\sum_{k=N'+1}^{\infty}\left(\|x^{k-1}-z^k\|+\|w^k-x^k\|\right)\\
&\leqslant \frac{2C_3}{D_3}\varphi_3\left(\chi_{N'+1}\right)+\frac{1}{4}\left(\|x^{N'}-w^{N'}\|+\|z^{N'}-x^{N'-1}\|\right)< \infty,
\end{split}
\end{equation*}
which implies the convergence of $\{x^k\}$ and $\sum_{k=1}^{\infty}\|x^{k}-x^{k-1}\|<\infty$.
\end{proof}

\section{Numerical test}\label{sec6}
In this section, we perform numerical experiments to study the behaviors of IRL$_1e_1$, IRL$_1e_2$ and IRL$_1e_3$. All codes are written in Matlab, and the experiments are performed in Matlab 2015b on a 64-bit PC with an Intel(R) Core(TM) i7-4790 CPU (3.60GHz) and 32GB of RAM.

We consider the following log penalty regularized least squares problem \cite{GoZhLuHuYe13}:
\begin{align}\label{LOG}
  \min\ F_{\log}(x) := \frac{1}{2}\|Ax-b\|^{2}+\sum_{i=1}^{n}\left(\lambda\log(|x_{i}|+\epsilon)-\lambda\log\epsilon\right),
\end{align}
where $A\in \mathbb{R}^{m\times n}$, $b\in \mathbb{R}^{m}$, $\lambda > 0$, $\epsilon > 0$. This is a special case of \eqref{primal} with $f$ being the least squares loss function, $C=\mathbb{R}^n$ and $\phi(t) = \lambda\log(t+\epsilon)-\lambda\log\epsilon$. Thus, we deduce from Theorem~\ref{Extra2thm1} that the sequence generated by IRL$_1e_2$ for a choice of $\{\theta_k\}$ satisfying \eqref{alpha} clusters at a stationary point of \eqref{LOG}. In addition, one can check that the corresponding $H_1$ and $H_3$ are continuous subanalytic functions \cite[Section~6.6]{FaPa03}, and hence they are KL functions \cite[Theorem~3.1]{BoDaLe07}. Consequently, we know from Theorem~\ref{thm2} (resp., Theorem~\ref{thm2_extra3}) that if $\sup_k\beta_k < 1$ (resp., $\{\theta_k\}$ satisfies \eqref{assumptionOnTheta}), then the whole sequence generated by IRL$_1e_1$ (resp., IRL$_1e_3$) converges to a stationary point of \eqref{LOG}.

In our experiments below, we compare IRL$_1e_1$, IRL$_1e_2$ and IRL$_1e_3$ with two other state-of-the-art algorithms for solving \eqref{LOG}: the general iterative shrinkage and thresholding algorithm (GIST) \cite{GoZhLuHuYe13} and an adaptation of the iteratively reweighted $\ell_1$ algorithm \cite[Algorithm~7]{Lu2014} with nonmonotone line-search (IRL$_1{ls}$). We discuss the implementation details of these algorithms below.

{\bf IRL$_1e_1$}. For this algorithm, we set $L=\lambda_{\max}(AA^T)$ and choose $\{\beta_k\}$ as in FISTA\cite{BeTe09,Nes13} with both the adaptive and fixed restart schemes \cite{ODoCa15}:\footnote{In our experiments, this quantity is computed in matlab with code {\sf lambda=norm(A*A')}, when $ m<2000$ and by {\sf opts.issym = 1}; {\sf lambda = eigs(A*A',1,'LM',opts); otherwise.}}
\[
\beta_k=\theta_k(\theta_{k-1}^{-1}-1)\ {\rm with\ } \theta_{k+1}=\frac{2}{1 + \sqrt{1+4/\theta_k^2}} {\rm\ and\ } \theta_0=\theta_{-1}=1
\]
and we reset $\theta_{k-1}=\theta_k=1$ every $200$ iterations, or when $\left< y^{k-1}-x^k,x^k-x^{k-1}\right> > 0$. It is clear that $\sup_k \beta_k < 1$.
We initialize this algorithm at the origin and terminate it when
\[
\frac{2L\|x^{k+1}-y^k\| +\epsilon^{-2}\lambda\|x^{k+1} - x^k\|}{\max\left\{1,\|x^{k+1}\|\right\}}<10^{-4}.
\]
Notice from \eqref{optimal_1} that  this termination criterion implies that $\dist(0,\partial F(x^{k+1}))\leqslant 10^{-4}\max\left\{1,\|x^{k+1}\|\right\}$.

{\bf IRL$_1e_2$}. For this algorithm, we set $L=\lambda_{\max}(AA^T)$, and let $\theta_k$ be as in FISTA\cite{BeTe09,Nes13} for the first 50 iterations, i.e., $\theta_0=1$ and $\theta_{k+1}=\frac{2}{1 + \sqrt{1+4/\theta_k^2}}$ for $0\leqslant k\leqslant 48$, $\theta_{50} = \theta_{49}$, and we update $\theta_{k} = \theta_{99-k}$ for $51\leqslant k \leqslant 99$ and set $\theta_k=\theta_{{\sf mod}(k,100)}$ for $k\geqslant 100$. It can be verified with simple computation that this choice of $\{\theta_k\}$ satisfies \eqref{alpha}.
We initialize the algorithm at the origin and terminate it when
\[
\frac{L\|z^{k+1}-y^k\|+\epsilon^{-2}\lambda\|x^k-z^{k+1}\|+L\|x^{k+1}-y^k\|}{\max\left\{1,\|z^{k+1}\|\right\}}<10^{-4}.
\]
Observe from \eqref{stationOfSub2} that this termination criterion implies that $\dist(0,\partial F(z^{k+1}))\leqslant 10^{-4}\max\left\{1,\|z^{k+1}\|\right\}$.

{\bf IRL$_1e_3$}. For this algorithm, we set $L=\lambda_{\max}(AA^T)$, and we generate a sequence $\{\rho_k\}$ as in FISTA in the first $57$ iterations and fix it from then on, i.e., $\rho_{0}=1$ and
\[\rho_{k+1}=
\begin{cases}
  \frac{2}{1 + \sqrt{1+4/\rho_k^2}}  &0\leqslant k \leqslant 55,\\
  \rho_{56}  &k\geqslant 56.
\end{cases}
\]
We then set $\theta_k = \rho_{k+6}$ for all $k\ge 0$.
It can be verified that the above $\{\theta_k\}$ satisfies \eqref{assumptionOnTheta} with $\gamma = 0.95$.
We initialize the algorithm at the origin and terminate it when
\[
\frac{2L\|x^{k+1}-y^k\| +\epsilon^{-2}\lambda\|x^{k+1} - x^k\|}{\max\left\{1,\|x^{k+1}\|\right\}}<10^{-4}.
\]
Note from \eqref{converge33} that  this termination criterion implies that $\dist(0,\partial F(x^{k+1}))\leqslant 10^{-4}\max\left\{1,\|x^{k+1}\|\right\}$.

\textbf{GIST.} This algorithm was proposed in \cite{GoZhLuHuYe13}; see also \cite{WrNoFi09}. Following the notation in \cite[Appendix A, Algorithm~1]{ChLuPo12}, here we set $c = 10^{-4}, \tau=2, M = 4$ and set $L_0^0=1$ and
\begin{equation}\label{BBstep}
L_k^0=\min\left\{10^8, \max\left\{\frac{\|A(x^k-x^{k-1})\|^2}{\|x^k-x^{k-1}\|^2}, 10^{-8}\right\}\right\}
\end{equation}
for $k\geqslant 1$. Note that the subproblem in \cite[Appendix A, (A.4)]{ChLuPo12} now becomes
\[
\min\limits_{x\in\mathbb{R}^n}\left\{\left<A^T(Ax^k-b),x-x^k\right>+\frac{L_k}{2}\|x-x^k\|^2+\sum_{i=1}^{n}\left(\lambda\log(|x_{i}|+\epsilon)-\lambda\log\epsilon\right)\right\}
\]
whose closed form solution can be found in \cite{GoZhLuHuYe13}. We initialize the algorithm at the origin and terminate it when
\[
\frac{\|\nabla f(x^k)-\nabla f(x^{k+1})\| + L_k\|x^k-x^{k+1}\|}{\max\left\{1,\|x^{k+1}\|\right\}}<10^{-4};
\]
this condition implies that $\dist(0,\partial F(x^{k+1}))\leqslant 10^{-4}\max\left\{1,\|x^{k+1}\|\right\}$.

\textbf{IRL$_1{ls}$.} This algorithm is an adaptation of \cite[Algorithm~7]{Lu2014}, which was originally designed for solving \eqref{primal} with $\phi(t)=\lambda\min\{p(ts-\frac{s^q}{q}):\;{0\leqslant s\leqslant(\frac{\epsilon}{\lambda n})^{\frac{1}{q}}}\}$ for some $p\in(0,1)$, $q = \frac{p}{p-1}$, $\lambda>0$ and $\epsilon>0$. For ease of presentation, we present our adaptation below in Algorithm~\ref{alg_adap}. Its convergence can be proved by adapting the convergence analysis of \cite[Algorithm~7]{Lu2014} and that of \cite{WrNoFi09}.

\begin{algorithm}
\caption{\bf Iteratively reweighted $\ell_1$ algorithm with nonmonotone line-search for \eqref{LOG} (IRL$_1{ls}$)}\label{alg_adap}
\begin{algorithmic}
\item[Step 0.] Let $0 < L_{\min} < L_{\max}$, $\tau> 1$ and $c > 0$ be given. Input an initial point $x^0$ and set $k = 0$.
\item[Step 1.]Choose $L_{k}^{0}\in [L_{\min},L_{\max}]$ and set $L_{k}=L_{k}^{0}$.
\item[Step 2.] Set
\[
\begin{cases}
&\displaystyle s_i^{k+1}= \frac{\lambda}{|x_i^{k}|+\epsilon}\ {\rm for}\ i=1,\ldots,n;\\
&\displaystyle x^{k+1}=\argmin_{y\in C} \left\{\left<\nabla f(x^k), y-x^k\right> + \frac{L_k}{2}\|y-x^k\|^2 + \sum_{i=1}^ns_i^{k+1}|y_i|\right\}.\\
\end{cases}
\]
\item[Step 3.]If
\[
F_{\log}(x^{k+1})> \max_{[k-M]_+\leqslant s\leqslant k}F_{\log}(x^s)  - \frac{c}{2}\|x^{k+1}-x^{k}\|^{2},
\]
let $L_{k}=\tau L_{k}$, and go to Step 2.
\item[Step 4.] If a termination criterion is not met, set $k = k+1$ and go to Step 1.
\end{algorithmic}
\end{algorithm}

For this algorithm, we let $ L_{\min} = 10^{-8}, L_{\max }= 10^{8}, c = 10^{-4}, \tau=2, M = 4$ and set $L_0^0=1$ and for $k\geqslant 1$, we set $L^0_k$ as in \eqref{BBstep}.
We initialize the algorithm at the origin and terminate it when
\[
\frac{\|\nabla f(x^k)-\nabla f(x^{k+1})\| + (L_k +\epsilon^{-2}\lambda)\|x^k-x^{k+1}\|}{\max\left\{1,\|x^{k+1}\|\right\}}<10^{-4}.
\]
From Step 2 in Algorithm~\ref{alg_adap}, one can observe that this termination criterion implies that $\dist(0,\partial F(x^{k+1}))\leqslant 10^{-4}\max\left\{1,\|x^{k+1}\|\right\}$ at termination.

We compare the above algorithms on random instances. We first generate an $m\times n$ matrix $A$ with i.i.d. standard Gaussian entries and then normalize this matrix to have unit column norms. A subset $T$
of size $r=[\frac{m}{9}]$ is then chosen uniformly at random from $\{1, 2, 3, . . ., n\}$ and an $r$-sparse vector $y\in \mathbb{R}^m$ supported on $T$ with i.i.d.
standard Gaussian entries is generated. We then set $b=Ay+0.01\cdot\omega$, where $\omega\in \mathbb{R}^m $ has i.i.d. standard Gaussian entries.

In our experiments, we set $(m,n) = (720i,2560i)$, with $i=1,\ldots, 10$. We pick $\lambda=5\times 10^{-4}$ in \eqref{LOG} and experiment with $\epsilon=0.1$ and $0.5$. We present the corresponding results in Tables~\ref{0.5} and \ref{0.1}, respectively, where we report the time for computing $\lambda_{\max}(A^T A)$ ($t_0$), the CPU times in seconds (time) and the function values at termination (fval), averaged over $20$ random instances. One can see that our algorithms are usually faster than GIST and IRL$_1{ls}$ and return slightly better function values at termination. Moreover, IRL$_1e_1$ and IRL$_1e_3$ are usually faster than IRL$_1e_2$.

\begin{landscape}
 \begin{table}[H]
 \caption{$\lambda= 5e-4,\epsilon=0.5$}\label{0.5}
\scriptsize{
\begin{tabular}{|c|c|c|ccccc||ccccc|}
  \hline
 \multicolumn{2}{|c|}{\bf Problem Size} & & \multicolumn{5}{c||}{\bf time} & \multicolumn{5}{c|}{\bf fval} \\
 $m$  &$n$  &    $t_{0}$   &GIST&IRL$_1{ls}$ &IRL$_1e_1$  &IRL$_1e_2$ &IRL$_1e_3$&GIST&IRL$_1{ls}$ &IRL$_1e_1$  &IRL$_1e_2$ &IRL$_1e_3$
 \\\hline
   720 &  2560 &   0.1 &      1.7&   1.5&   0.7&   0.9&   0.6 & 3.7918e-02& 3.7918e-02& 3.7897e-02& 3.7900e-02& 3.7896e-02\\
  1440 &  5120 &   0.7 &      7.0&   6.6&   3.3&   4.3&   2.6 & 7.5904e-02& 7.5900e-02& 7.5859e-02& 7.5867e-02& 7.5858e-02\\
  2160 &  7680 &   0.7 &     15.0&  14.3&   7.2&   9.4&   5.6 & 1.1443e-01& 1.1443e-01& 1.1436e-01& 1.1437e-01& 1.1436e-01\\
  2880 & 10240 &   1.3 &     25.9&  25.1&  12.6&  16.5&   9.9 & 1.5224e-01& 1.5224e-01& 1.5215e-01& 1.5217e-01& 1.5215e-01\\
  3600 & 12800 &   2.4 &     39.4&  38.7&  19.8&  25.6&  15.4 & 1.8805e-01& 1.8805e-01& 1.8794e-01& 1.8796e-01& 1.8794e-01\\
  4320 & 15360 &   3.8 &     57.0&  55.9&  28.2&  36.4&  21.9 & 2.2774e-01& 2.2773e-01& 2.2761e-01& 2.2763e-01& 2.2761e-01\\
  5040 & 17920 &   6.3 &     76.8&  75.9&  38.8&  49.1&  30.0 & 2.6491e-01& 2.6491e-01& 2.6474e-01& 2.6478e-01& 2.6475e-01\\
  5760 & 20480 &   8.0 &     99.2&  98.5&  49.7&  62.9&  38.7 & 3.0627e-01& 3.0627e-01& 3.0609e-01& 3.0614e-01& 3.0609e-01\\
  6480 & 23040 &  11.3 &    126.1& 125.0&  63.8&  81.8&  49.7 & 3.4231e-01& 3.4233e-01& 3.4212e-01& 3.4216e-01& 3.4212e-01\\
  7200 & 25600 &  15.0 &    159.8& 157.3&  80.5& 102.7&  62.4 & 3.8133e-01& 3.8133e-01& 3.8111e-01& 3.8116e-01& 3.8111e-01\\
\hline
\end{tabular}   }
\end{table}

 \begin{table}[H]
 \caption{$\lambda= 5e-4,\epsilon=0.1$}\label{0.1}
\scriptsize{
\begin{tabular}{|c|c|c|ccccc||ccccc|}
\hline
 \multicolumn{2}{|c|}{\bf Problem Size} & & \multicolumn{5}{c||}{\bf time} & \multicolumn{5}{c|}{\bf fval} \\
 $m$  &$n$  &    $t_{0}$   &GIST&IRL$_1{ls}$ &IRL$_1e_1$  &IRL$_1e_2$ &IRL$_1e_3$&GIST&IRL$_1{ls}$ &IRL$_1e_1$  &IRL$_1e_2$ &IRL$_1e_3$
 \\\hline
   720 &  2560 &   0.1 &      0.6&   0.5&   0.3&   0.4&   0.3 & 9.3307e-02& 9.3307e-02& 9.3305e-02& 9.3302e-02& 9.3303e-02\\
  1440 &  5120 &   0.7 &      2.2&   2.0&   1.4&   1.8&   1.5 & 1.8663e-01& 1.8663e-01& 1.8663e-01& 1.8662e-01& 1.8662e-01\\
  2160 &  7680 &   0.7 &      4.6&   4.3&   3.0&   4.0&   3.3 & 2.7940e-01& 2.7940e-01& 2.7939e-01& 2.7938e-01& 2.7938e-01\\
  2880 & 10240 &   1.4 &      8.0&   7.6&   5.3&   7.0&   5.7 & 3.7401e-01& 3.7401e-01& 3.7401e-01& 3.7400e-01& 3.7400e-01\\
  3600 & 12800 &   2.5 &     12.4&  11.7&   8.3&  11.1&   9.1 & 4.6537e-01& 4.6537e-01& 4.6536e-01& 4.6535e-01& 4.6535e-01\\
  4320 & 15360 &   3.8 &     17.4&  16.7&  11.6&  15.3&  12.6 & 5.6347e-01& 5.6347e-01& 5.6346e-01& 5.6344e-01& 5.6344e-01\\
  5040 & 17920 &   6.2 &     23.1&  21.8&  15.5&  20.7&  16.9 & 6.5319e-01& 6.5320e-01& 6.5318e-01& 6.5316e-01& 6.5317e-01\\
  5760 & 20480 &   7.9 &     29.9&  28.3&  20.2&  26.8&  22.0 & 7.4930e-01& 7.4931e-01& 7.4929e-01& 7.4927e-01& 7.4927e-01\\
  6480 & 23040 &  10.9 &     38.1&  36.6&  25.7&  34.2&  28.1 & 8.3961e-01& 8.3961e-01& 8.3960e-01& 8.3957e-01& 8.3957e-01\\
  7200 & 25600 &  14.5 &     48.8&  46.2&  32.1&  42.7&  35.0 & 9.3472e-01& 9.3472e-01& 9.3471e-01& 9.3467e-01& 9.3468e-01\\
\hline
\end{tabular} }
\end{table}

\end{landscape}

\section{Concluding remarks}\label{sec7}

In this paper, we study how popular extrapolation techniques such as those presented in \cite{AuTe06,BeTe09,LaLuMo11,Nes13} can be incorporated into the iteratively reweighted $\ell_1$ scheme for solving \eqref{primal}. We proposed three versions of such algorithms, namely IRL$_1e_1$, IRL$_1e_2$ and IRL$_1e_3$, and analyzed their convergence under suitable assumptions on the extrapolation parameters. Convergence of the whole sequence is also established for IRL$_1e_1$ and IRL$_1e_3$ under additional Lipschitz differentiability assumption on $\phi$ and KL assumption on some suitable potential functions.

One future research direction is to look at other extrapolation techniques such as those used in the third APG method in \cite{Tse10}. It is also interesting to investigate whether the analysis in \cite{DrPa16,GhLa16} can be adapted to analyze the iteration complexity of IRL$_1e_1$, IRL$_1e_2$ and IRL$_1e_3$.

\end{document}